\documentclass[10pt,a4paper]{article}
\usepackage{amsmath, amsfonts,amsthm,amssymb,amsbsy,upref,color,graphicx,amscd,hyperref,makeidx,enumerate}
\usepackage[active]{srcltx}
\usepackage[latin1]{inputenc}

\newcommand{\sm}{\setminus}

\newcommand{\g}{\gamma}
\newcommand{\Om}{\Omega}

\newcommand{\lb}{\lambda}
\newcommand{\vps}{\varepsilon}

\newcommand{\ra}{\rightarrow}

\newcommand{\nif}{{n \rightarrow \infty}}

\newcommand{\sq}{\subseteq}

\def\ds{\displaystyle}
\def\eps{{\varepsilon}}
\def\N{\mathbb{N}}
\def\O{\Omega}
\def\om{\omega}
\def\R{\mathbb{R}}
\def\Q{\mathbb{Q}}

\def\HH{\mathcal{H}}

\def\Dr{D}

\newcommand{\be}{\begin{equation}}
\newcommand{\ee}{\end{equation}}
\newcommand{\bib}[4]{\bibitem{#1}{\sc#2: }{\it#3. }{#4.}}

\newcommand{\cp}{\mathop{\rm cap}\nolimits}

\numberwithin{equation}{section}
\theoremstyle{plain}
\newtheorem{teo}{Theorem}[section]
\newtheorem{lemma}[teo]{Lemma}
\newtheorem{cor}[teo]{Corollary}
\newtheorem{prop}[teo]{Proposition}
\newtheorem{deff}[teo]{Definition}
\theoremstyle{remark}
\newtheorem{oss}[teo]{Remark}

\title{Multiphase shape optimization problems}

\author{Dorin Bucur, Bozhidar Velichkov}

\begin{document}
\maketitle
\begin{abstract}
This paper is devoted to the analysis of multiphase shape optimization problems, which can formally be written as
$$\min\Big\{{g}(F_1(\Om_1),\dots,F_h(\Om_h))+ m|\ds \bigcup_{i=1}^h\Om_i| :\ \Om_i\subset D,\ \Om_i\cap \Om_j =\emptyset\Big\},$$
where $D\sq \R^d$ is a given bounded open set, $|\Om_i|$ is the Lebesgue measure of $\Om_i$ and $m$ is a positive constant. 
For a large class of such functionals, we analyse qualitative properties of the cells and the interaction between them. Each cell is itself subsolution for a (single-phase) shape optimization problem, from which we deduce properties like
finite perimeter, inner density,
separation by open sets, absence of triple junction points, etc. 
 As main examples we consider functionals involving the eigenvalues of the Dirichlet Laplacian of each cell, i.e. $F_i=\lambda_{k_i}$.
\end{abstract}

\textbf{Keywords:} shape optimization, multiphase, eigenvalues, optimal partitions

\textbf{2010 Mathematics Subject Classification:} 49J45, 49R05, 35P15, 47A75, 35J25

\section{Introduction}\label{intro}
Let $D \sq \R^d$ be a bounded open set and $m\ge 0$. We study multiphase shape optimization problems of the form 
\begin{equation}\label{sintroe0}
\min\Big\{{g}(F_1(\Om_1),\dots,F_h(\Om_h))+ m|\ds \bigcup_{i=1}^h\Om_i|  :\ \Om_i\subset D,\ \Om_i\cap \Om_j =\emptyset\Big\},
\end{equation}
where  $|\Om|$ is the Lebesgue measure of $\Om$. To each cell $\Om_i$, we associate a shape functional $F_i$, the interaction between cells being described by the function $g:\R^h \ra \R$. If one fixes $h-1$ cells of an optimal configuration, and let formally free only one, this cell is a shape subsolution. In a neighborhood of the junction points, it can be compared only with its inner perturbations.  One of the main questions raised by such a shape optimization problem concerns precisely  the interaction between the cells. The functionals ${ F_i}$ we consider here, involve quantities related to the Dirichlet Laplacian operator on each cell as for example the eigevalues  $ (\lb_k(\Om_i))_{k\in \N}$ of the Laplace operator with Dirichlet boundary conditions on a quasi-open set $\Om_i$.

For  a very particular choice of $g$ and $F_i$, this topic was intensively studied in the last years, essentially for functionals involving the first eigenvalue
\begin{equation}\label{bvjkl}
{g}(F_1(\Om_1),\dots,F_h(\Om_h))=  \sum_{i=1}^h \lb_{1} (\Om_i)\;\;\mbox{and}\;\; {g}(F_1(\Om_1),\dots,F_h(\Om_h))=\max_{i=1,\dots,h} \lb_{1} (\Om_i).
\end{equation}
For $m=0$, we refer the reader to the papers \cite{coteve03,coteve051,coteve052,hehote,caflin}  
and the references therein, while 
for $m>0$, only the case  $h=1$ was studied in \cite{bhp05,brla}. 

Many interesting qualitative results were obtained for \eqref{bvjkl}, among which regularity properties of the boundaries and interesting information on the junction points. 

In this paper we intend to discuss general functionals $F_i$, precisely functionals which have a variation  controlled by the Dirichlet energy (see Definition \ref{gammalip} below) e.g.  the $k$-th eigenvalue of the Dirichlet Laplacian, in a context where  the measure constraint is relevant ($m>0$). For example,  problems of the form
\begin{equation}\label{buveq00}
\min \Big\{\sum_{i=1}^n \lb_{k_i} (\Om_i)+m|\Om_i| :\ \Om_i\subset D,\ \Om_i \ \hbox{quasi-open,}\;\ \Om_i \cap \Om_j =\emptyset\Big\}. 
\end{equation}
fit in our framework. 
If $m>0$, the sets $\Om_i$ will not in general cover $D$ and a void region will appear, so the solution will be a sort of lacunary partition of $D$.
As we consider general functionals $F_i$,   the same tools used for the regularity of the free boundaries in \cite{coteve03,caflin} can not be adapted. Even if $F_i$ is simply the $k$-th eigenvalue of the Dirichlet Laplacian, obtaining a regularity results is a  complicated task, since the $k$-th eigenvalue is itself a critical point and not a minimizer as the first eigenvalue is. 

 We refer the reader to the survey papers \cite{but10, hen03} and the books \cite{bubu05, hen06, hepi05} for a detailed introduction to the topic of shape optimization problems. 
Existence of a solution for \eqref{sintroe0} in the class of quasi-open sets was proved in \cite{bubuhe98}  and is a consequence of a general result due to Buttazzo and Dal Maso (see \cite{budm91, budm93}).

We focus in this paper on the analysis of the geometric interaction between cells. 
Our main tool involves  the analysis of the shape subsolutions for the torsional energy, i.e. quasi-open  sets $\Omega\subset\R^d$ which satisfy for some $m >0$
\begin{equation}\label{sintroe1}
E(\Omega)+m|\Omega|\le E(\widetilde\Omega)+m|\widetilde\Omega|,\ \forall\widetilde\Omega\subset\Omega,
\end{equation}
where $E(\Omega)$ is the torsional energy (see also \eqref{diren} below)
$$E(\Om)= \min\Big\{\frac12 \int_\Om |\nabla u|^2\,dx - \int _\Om u\,dx : u \in H^1_0(\Om)\Big\}.$$
Under mild assumptions on $g$ and for a quite large class of functionals $F_i$, every cell of the optimal solution of   \eqref{sintroe0} is a shape subsolution of the torsional energy.

 Analyzing the properties of the  subsolutions we prove that (Sections \ref{s4} and \ref{s5})
 \begin{itemize}
\item each cell satisfies inner density estimates and has finite perimeter;
\item there are no triple junction points, i.e.  $\partial\Omega_i\cap\partial\Omega_j\cap\partial\Omega_k=\emptyset$, for different $i,j,k$;

\item each (quasi-open) cell $\Om_i$ can be isolated by an open set $D_i$ from the other cells, and solves
the problem
$$\min\{F_i(\Om) :\ \Om \sq D_i,\ \Om \; \mbox{quasi-open},\;|\Om|=|\Om_i| \};$$
\item if $F_i$ depends in \eqref{sintroe0} only on the first and the second eigenvalues, there exists a solution consisting of open cells;

\item in $\R^2$, for $m=0$,   every solution of \eqref{buveq00} is equivalent to a solution consisting of open sets.
\end{itemize}
We emphasize that a subsolution is not, in general, an open set, as  Remark  \ref{buveqo1} shows. Even for the solutions of some simple one-phase shape optimization problems, as 
\begin{equation}\label{ex011.1}
\min\{\lb_{k}(\Om) :\ \Om \sq D,\ \Om\ \hbox{quasi-open},\ |\Om|=m\},
\end{equation}
 with $k \ge 3$,
  the optimal set $\Omega$ is, a priori, no more than a quasi-open set. Until recently, the only functionals which were known to have (smooth) open sets as solutions were the first eigenvalue (see \cite{brla}) and the Dirichlet Energy (see \cite{bhp05}). 

The study of triple junction points goes through a multiphase monotonicity formula (Lemma \ref{mth3})  in the spirit of \cite{cajeke} and \cite[Lemmas 4.2 and 4.3]{coteve03}, which is proved in the Appendix. Precisely, if $u_i \in H^1(B_1)$, $i=1,2,3$, are three non-negative functions with disjoint supports and such that $\Delta u_i\ge-1$, for each $i=1,2,3$, then there are dimensional constants $\eps>0$ and $C_d>0$ such that for each $r\in(0,\frac12)$
\begin{equation}\label{mth3e1}
\prod_{i=1}^3\left(\frac{1}{r^{2+\eps}}\int_{B_r}\frac{|\nabla u_i|^2}{|x|^{d-2}}\,dx\right)\le C_d\left(1+\sum_{i=1}^3\int_{B_1}\frac{|\nabla u_i|^2}{|x|^{d-2}}\,dx\right).
\end{equation}  
The main gain of this multiphase monotonicity formula is that for junction points of three cells (or more), at least one gradient decays faster than $r^{\vps/2}$, which contradicts the super linear decay which is expected for subsolutions, cf. Lemma \ref{bnd}.

\section{Preliminaries}\label{pr}
In this section we recall some of the notions and results that we need in this paper. 
\subsection{Capacity and quasi-open sets} As we mentioned in the introduction, for our purposes it is convenient to extend the notion of a Sobolev space and Laplace operator to measurable sets. One has to use the notion of capacity of a set $E\subset\R^d$, which is defined as

\begin{equation}
\cp(E)=\inf\left\{\|u\|_{H^1}:\ u\in H^1(\R^d),\ u\ge 1\ \hbox{in a neighbourhood of}\ E\right\},
\end{equation}
where $\|u\|^2_{H^1}=\|u\|^2_{L^2}+\|\nabla u\|^2_{L^2}$ (see, for example, \cite{hepi05} for more details). 
\begin{itemize}
\item We say that a property $\mathcal{P}$ holds quasi-everywhere (shortly q.e.) in $\R^d$, if the set of points $E$, where $\mathcal{P}$ does not hold, is of zero capacity ($\cp(E)=0$). 
\item We say that a set $\Omega\subset\R^d$ is quasi-open, if for each $\varepsilon>0$ there is an open set $\omega_\varepsilon$ of capacity $\cp(\omega_\varepsilon)\le\varepsilon$ such that $\Omega\cup\omega_\varepsilon$ is an open set. 
\item A function $u:\R^d\to\R$ is quasi-continuous, if for each $\varepsilon>0$ there is an open set $\omega_\varepsilon$ of capacity $\cp(\omega_\varepsilon)\le\varepsilon$ such that the restriction of $u$ on the closed set $\R^d\setminus\omega_\varepsilon$ is a continuous function.
\end{itemize} 
We note that any function $u\in H^1(\R^d)$ has a quasi-continuous representative $\widetilde u:\R^d\to\R$, which is unique up to sets of zero capacity (see \cite{hepi05}). Moreover, if the sequence $u_n\in H^1(\R^d)$ converges strongly in $H^1(\R^d)$ to the function $u\in H^1(\R^d)$, then there is a subsequence converging quasi-everywhere. 

The notion of capacity and all the properties mentioned above can be naturally extended on the $(d-1)$ dimensional sphere $\partial B_1$, which locally behaves like $\R^{d-1}$.  In particular, a set $E \sq \partial B_1$ has $(d-1)$-capacity zero, if when seen through any local chart, it has zero capacity in $\R^{d-1}$.
A set of $(d-1)$-capacity zero on the sphere has also zero capacity as a subset of $\R^d$. As well, let $E\subset \partial B_1$ be a set of non-zero $(d-1)$-capacity in $\partial B_1$. Then there is a constant $C>0$ (related to the first Dirichlet eigenvalue of the Laplace Beltrami operator on $\partial B_1\setminus E$) such that for each $u\in H^1(\partial B_1)$, which vanishes $(d-1)$-quasi-everywhere on $E$, we have  
$$\int_{\partial B_1}u^2\,d\HH^{d-1} \le C\int_{\partial B_1}|\nabla_\tau u|^2\,d\HH^{d-1}.$$

In $\R^d$ the canonical quasi-continuous representative $\widetilde u$ of $u\in H^1(\R^d)$ has a pointwise definition, i.e. for quasi-every $x\in\R^d$ the following limit exists
\begin{equation}\label{pwdefu}
\widetilde u(x)=\lim_{r\to0}\frac{1}{|B_r|}\int_{B_r(x)} u (y)\,dy.
\end{equation}

   We now define the Sobolev space $H^1_0(\Omega)$, for every measurable set $\Omega\subset\R^d$, 
\begin{equation}
H^1_0(\Omega)=\left\{u\in H^1(\R^d):\ \widetilde u=0\ \,\hbox{q.e. on}\ \Omega^c\right\}.
\end{equation}
In the case when $\Omega$ is an open set, $H^1_0(\Omega)$ coincides with the classical Sobolev space defined as the closure of the smooth functions with compact support $C^\infty_c(\Omega)$, with respect to the norm $\|\cdot\|_{H^1}$ (see \cite{hepi05}). We note that the quasi-open sets are the natural domains for the Sobolev spaces. Indeed, for any measurable set $\Omega$, there is a quasi-open set $\omega\subset\Omega$ q.e. such that $H^1_0(\omega)=H^1_0(\Omega)$ and which is also the largest quasi-open set contained q.e. in $\Omega$. In the case when $\Omega$ has finite measure, the set $\omega$ coincides quasi-everywhere with the level set $\{w_\Omega>0\}$, where $w_\Omega\in H^1(\R^d)$ is the weak solution of 
\begin{equation}\label{w}
-\Delta w_\Omega=1,\qquad w_\Omega\in H^1_0(\Omega),
\end{equation}
defined as the unique minimizer in $H^1_0(\Omega)$ of the torsional functional 
\begin{equation}\label{J}
J(w)=\frac12\int_{\R^d} |\nabla w|^2\,dx-\int_{\R^d} w\,dx.
\end{equation}
In particular, if $\Omega$ is a quasi-open set,  the strong maximum principle holds in the form $\Omega=\{w_\Omega>0\}$ q.e. The torsional energy $E(\Omega)$ of the quasi-open set of finite measure $\Omega\subset\R^d$ is defined as 
\begin{equation}\label{diren}
E(\Omega)=-\frac12\int_\Omega w_\Omega\,dx.
\end{equation}

\subsection{The $\gamma$ and weak $\gamma$-convergence}
The identification of the quasi-open sets $\Omega$ and their torsional function $w_\Omega$ leads naturally to the following (more functional than geometrical) distance.  We define the so called $\g$-distance between two quasi-open sets of finite measure $\Om_1$ and $\Om_2$ by

$$d_\g(\Om_1,\Om_2)=\int_{\R^d}|w_{\Om_1}-w_{\Om_2}|\,dx.$$
Notice that if $\Om_1 \sq \Om_2$ then $d_\g(\Om_1,\Om_2)=\frac12 [E(\Om_1)-E(\Om_2)]$.

\begin{deff}
In the family of quasi-open sets of finite measure, it is said that the sequence $\Omega_n$ $\gamma$-converges to $\Omega$ if  $d_\g(\Om_n, \Om)\ra 0$, as $\nif$.
\end{deff}
Sometimes, the $\g$-distance is defined using the $L^2$-norm of $w_{\Om_1}-w_{\Om_2}$. In a family of  sets  with uniformly bounded measure,  the two distances are equivalent. For the purposes of our paper, it is more convenient to use  the $L^1$-norm. 
\begin{deff}
In the family of quasi-open sets of finite measure, it is said that the sequence $\Omega_n$ weak $\gamma$-converges to $\Omega$  if the sequence of the corresponding torsional functions $w_{\Omega_n}$ converges in $L^2(\R^d)$ to some function $w\in H^1(\R^d)$ and $\Omega=\{w>0\}$. 
\end{deff}

\begin{oss}\label{wg130312}
For every (quasi-) open set $\Dr$ of finite Lebesgue measure the set 
$$\mathcal{A}_{cap} (D)=\left\{\Omega:\ \Omega\ \hbox{quasi-open},\ \Omega\subset\Dr\right\},$$ 
is sequentially compact for the weak $\gamma$-convergence. Indeed, let $\Omega_n\in\mathcal{A}_{cap} (D)$ be a sequence of quasi-open sets and let $w_n$ be the sequence of corresponding torsional functions. By \eqref{w} and the Gagliardo-Nirenberg-Sobolev inequality, we have
$$\int_\Dr |\nabla w_n|^2\,dx=\int_\Dr w_n\,dx\le |\Omega_n|^{\frac{d+2}{2d}}\|w_n\|_{L^{\frac{2d}{d-2}}}\le C_d|\Dr|^{\frac{d+2}{2d}}\|\nabla w_n\|_{L^2},$$ 
and so, $w_n$ is bounded in $H^1_0(\Dr)$. The compactness of $\mathcal{A}_{cap} (D)$ now follows by the compactness of the inclusion $H^1_0(\Dr)\subset L^2(\Dr)$. 
\end{oss}
\begin{oss}\label{gwg130212}
As a consequence of the Fatou Lemma, 
the Lebesgue measure is lower semi-continuous with respect to the weak $\gamma$-convergence in $\mathcal{A}_{cap} (D)$.
Moreover, if the sequence $\Omega_n\in\mathcal{A}_{\cp}(\Dr)$ weak $\gamma$-converges to $\Omega$, then, for a suitable subsequence, there is a sequence of quasi-open sets $\omega_k$ such that $\omega_k\supset\Omega_{n_k}$ and $\omega_k$ $\gamma$-converges to $\Omega$ (see for example \cite{bubu05}).
\end{oss}
   The weak $\gamma$-convergences is used to establish existence results for shape optimization problems where the shape functional is $\g$-continuous and decreasing for inclusions. We recall here a general existence result, proved in \cite{bubuhe98}, which is a multiphase version of the classical Butazzo-Dal Maso Theorem (see \cite{budm93}).
\begin{teo}\label{introth130312}
Let $\Dr\subset\R^d$ be a quasi-open set of finite Lebesgue measure and let $\mathcal{F}:\left[\mathcal{A}_{\cp}(\Dr)\right]^h\to\R$  satisfy
\begin{enumerate}[(i)]
\item $\mathcal{F}$ is decreasing with respect to the inclusion, i.e. if $\widetilde\Omega_i\subset\Omega_i$, for all $i=1,\dots,h$, then
$$\mathcal{F}(\Omega_1,\dots,\Omega_h)\le \mathcal{F}(\widetilde\Omega_1,\dots,\widetilde\Omega_h);$$
\item $\mathcal{F}$ is lower semi-continuous with respect to the $\gamma$-convergence, i.e. if $\Omega_i^n$ $\gamma$-converges to $\Omega_i$, for every $i=1,\dots,h$, then 
$$\mathcal{F}(\Omega_1,\dots,\Omega_h)\le\liminf_{n\to\infty}\mathcal{F}(\Omega_1^n,\dots,\Omega_h^n).$$
\end{enumerate}
Then the multiphase shape optimization problem 
\begin{equation}\label{sopF130312}
\min\left\{\mathcal{F}(\Omega_1,\dots,\Omega_h)+m\sum_{i=1}^h|\Omega_i|:\ \Omega_i\in\mathcal{A}_{\cp}(\Dr),\ \forall i;\ \Omega_i\cap\Omega_j=\emptyset,\ \forall i\neq j\right\},
\end{equation}
has a solution for every $m\ge0$. 
\end{teo}
The proof is a consequence of Remarks \ref{wg130312} and \ref{gwg130212}, the essential point being the fact that a decreasing shape functional which is $\g$-lower semicontinuous, is also weak $\g$-lower semicontinuous. 

\begin{oss}
There is a large class of functionals which are known to be decreasing and lower semi-continuous with respect to the $\gamma$-convergence (see \cite{bubu05, but10}, for more details). Typical examples are
\begin{itemize}
\item the Dirichlet Energy  defined as
$$\min\left\{\frac12\int_\Omega |\nabla u|^2\,dx-\int_\Omega uf\,dx:\ u\in H^1_0(\Omega)\right\},$$
where $f\in L^2(D)$ is a given function;
\item the eigenvalues of the Dirichlet Laplacian, i.e. 
$$\lambda_k(\Omega)=\min_{S_k\subset H^1_0(\Omega)}\max\left\{\int_\Omega |\nabla u|^2\,dx:\ u\in S_k,\ \int_\Omega u^2\,dx=1\right\},$$
where the minimum is over all $k$-dimensional subspaces $S_k\subset H^1_0(\Omega)$. 
\end{itemize}
\end{oss}

\subsection{Measure theoretic tools}

We shall use throughout the paper the notions of a measure theoretic closure $\overline{\Omega}^M$ and a measure theoretic boundary $\partial^M\Omega$ of a Lebesgue measurable set $\Omega\subset\R^d$, which are defined as:
\begin{equation*}
\overline{\Omega}^M=\left\{x\in\R^d:\ |B_r(x)\cap\Omega|>0,\ \forall r>0\right\},
\end{equation*}

\begin{equation*}
\partial^M\Omega=\left\{x\in\R^d:\ |B_r(x)\cap\Omega|>0,\ |B_r(x)\cap\Omega^c|>0,\ \forall r>0\right\}.
\end{equation*}

Moreover, for every $0\le\alpha\le1$, we define the set of points of density $\alpha$ as 
\begin{equation*}
\Omega_{(\alpha)}=\left\{x\in\R^d:\ \lim_{r\to0}\frac{|B_r(x)\cap\Omega|}{|B_r|}=\alpha\right\}.
\end{equation*}
If $\Omega$ has finite perimeter in sense of De Giorgi, i.e. the distributional gradient $\nabla 1_\Omega$ is a measure of finite total variation $|\nabla 1_\Omega|(\R^d)<+\infty$,  the generalized perimeter of $\Omega$ is given by 
$$P(\Omega)=|\nabla 1_\Omega|(\R^d)=\HH^{d-1}(\partial^*\Omega),$$
where $\partial^*\Omega$ is the reduced boundary of $\Omega$. 

The $s$-dimensional Hausdorff measure is denoted by $\HH^s$. To simplify notations and when no ambiguity occurs, we shall use the notation $|\partial B_r(x)|$ for the $(d-1)$ Hausdorff measure of the boundary of the ball centered in $x$ of radius $r$. 

\begin{oss}\label{canrepqo}
We note that the quasi-open sets are defined up to a set of zero capacity. We may define a canonical representative of the quasi-open set $\Omega$ as $\Omega=\{\widetilde{w}_\Omega>0\}$, where $\widetilde{w}_\Omega$ is the quasi-continuous representative of $w_\Omega$ defined as $0$, on the non-Lebesgue points for $w_\Omega$, and as the limit \eqref{pwdefu}, on the Lebesgue points for $w_\Omega$. With this identification, we have that
\begin{itemize}
\item each point $x\in\Omega$ is a Lebesgue point for $w_\Omega$;
\item the measure theoretical and the topological closure of $\Omega$ coincide $\overline\Omega=\overline{\Omega}^M$;
\item if $\Omega_1$ and $\Omega_2$ are two disjoint quasi-open sets, i.e. $\cp(\Omega_1\cap\Omega_2)=0$, then the measure theoretical and the topological common boundaries coincide $$\partial\Omega_1\cap\partial\Omega_2=\overline\Omega_1\cap\overline\Omega_2=
\overline{\Omega}_1^M\cap\overline{\Omega}_2^M=\partial^M\Omega_1\cap\partial^M\Omega_2.$$
\end{itemize}
\end{oss}

\subsection{Monotonicity theorems}

  We recall the following two-phase monotonicity formula due to Caffarelli, Jerison and Kenig \cite{cajeke}. 
\begin{teo}\label{mth2} {\rm\bf (Caffarelli, Jerison, Kenig)}
Let $u_1,u_2\in H^1(B_1)$ be two non-negative  functions such that $\Delta u_i\ge-1$, for $i=1,2$, and $\int_{\R^d} u_iu_j\,dx=0$. Then there is a dimensional constant $C_d$ such that for each $r\in(0,\frac12)$ we have
\begin{equation}\label{mth2e1}
\prod_{i=1}^2\left(\frac{1}{r^{2}}\int_{B_r}\frac{|\nabla u_i|^2}{|x|^{d-2}}\,dx\right)\le C_d\left(1+\sum_{i=1}^2\int_{B_1}\frac{|\nabla u_i|^2}{|x|^{d-2}}\,dx\right).
\end{equation} \end{teo}

In \cite{cajeke},  Theorem \ref{mth2} was stated with the additional assumption that the functions $u_1$ and $u_2$ are continuous. An inspection of the original proof shows that this assumption is not necessary, as it will be seen in the proof of Lemma \ref{mth3}, in the Appendix.

The following monotonicity lemma is due to Conti, Terracini and Verzini and holds in two dimensions.
\begin{teo}\label{mth2.1} {\rm\bf (Conti, Terracini, Verzini)}
In $\R^2$, let $u_1,u_2,u_3 \in H^1(B_1)$ be three non-negative subharmonic functions such that  $\int_{\R^2} u_iu_j\,dx=0$. Then the function
\begin{equation}\label{mth2e1.1}
r\mapsto \prod_{i=1}^3\left(\frac{1}{r^{3}}\int_{B_r}|\nabla u_i|^2\,dx\right)
\end{equation} 
is nondecreasing on $[0,1]$.
\end{teo}

As in our problem the functions are not subharmonic, the argument we search is closer  to   Theorem \ref{mth2} than to Theorem  \ref{mth2.1}.  We give a multiphase monotonicity formula in the spirit of Theorem \ref{mth2}. We are not able to obtain optimal decreasing rates as in Theorem  \ref{mth2.1}, but the estimate below will be sufficient for our purposes and holds in any dimension of the space. 

\begin{lemma}[Three-phase monotonicity lemma]\label{mth3}
Let $u_i \in H^1(B_1)$, $i=1,2,3$, be three non-negative Sobolev functions such that $\Delta u_i\ge-1$, for each $i=1,2,3$, and $\int_{\R^d} u_iu_j\,dx=0$, for each $i\neq j$. Then there are dimensional constants $\eps>0$ and $C_d>0$ such that, for every $r\in(0,\frac12)$, we have 
\begin{equation}\label{mth3e1}
\prod_{i=1}^3\left(\frac{1}{r^{2+\eps}}\int_{B_r}\frac{|\nabla u_i|^2}{|x|^{d-2}}\,dx\right)\le C_d\left(1+\sum_{i=1}^3\int_{B_1}\frac{|\nabla u_i|^2}{|x|^{d-2}}\,dx\right).
\end{equation}   
\end{lemma}
    The proof of this result follows the main arguments and steps of Theorem \ref{mth2}. For the convenience of the reader, we report it in the Appendix, with an emphasis on the technical differences brought by the lack of continuity and the presence of the third phase.

\section{Shape subsolutions for the torsional energy}\label{de}  
   In this section we study the quasi-open sets of finite measure which are minimal for the functional $E(\cdot)+m|\cdot|$, with respect to internal variations of the domain. Sets satisfying this property will be called energy subsolutions. More precisely, we give the following:
\begin{deff}\label{sub}
We say that the quasi-open set $\Omega\subset\R^d$ is a shape subsolution for the torsional energy (or, simply, energy subsolution), if there are real constants $m>0$ and $\varepsilon>0$ such that for each quasi-open set $\tilde\Omega\subset\Omega$ for which $d_\g(\widetilde\Omega,\Omega)<\varepsilon$, we have 
\begin{equation}\label{sub1}
E(\Omega)+m|\Omega|\le E(\tilde\Omega)+m|\tilde\Omega|.
\end{equation}
\end{deff}   
Inequality \eqref{sub1} is equivalent to say $2m|\Om \sm \tilde \Om| \le d_\g(\widetilde\Omega,\Omega) $ if $d_\g(\widetilde\Omega,\Omega)<\varepsilon$.
\begin{oss}
If $\Omega$ is an energy subsolution with constant $m$ and $m'\le m$, then $\Omega$ is also an energy subsolution with constant $m'$.
\end{oss}
\begin{oss}
We recall that if $\Omega\subset\R^d$ is a quasi-open set of finite measure and $t>0$ is a given real number, then we have 
$$w_{t\Omega}(x)=t^2w_\Omega(x/t)\qquad \hbox{and}\qquad E(t\Omega)=t^{d+2}E(\Omega).$$
Thus, if $\Omega$ is an energy subsolution with constants $m$ and $\varepsilon$, then $\Omega'=t\Omega$ is an energy subsolution with constants $m'=1$ and $\varepsilon'=\varepsilon t^{d+2}$, where $t=m^{-1/2}$.
\end{oss}   

\begin{oss}
If the energy subsolution $\Omega\subset\R^d$ is smooth, then writing the optimality condition for local perturbations of the domain $\Omega$ with smooth vector fields (see, for example, \cite[Chapter 5]{hepi05}) we obtain that
$$|\nabla w_\Omega|^2(x)\ge 2m,$$
for each $x\in\partial\Omega$.
\end{oss}

The energy subsolutions play an important role in the study of the optimal domains even for very general spectral optimization problems. In fact, in \cite{bulbk} the following Theorem was proved:  
 
\begin{teo}\label{gammalipth}
Let $k>0$ and let $\Omega\subset\R^d$ satisfy
\begin{equation}
\lambda_k(\Omega)+m|\Omega|\le \lambda_k(\widetilde\Omega)+m|\widetilde\Omega|,
\end{equation}
for each quasi-open set $\widetilde\Omega\subset\Omega$ such that $d_\g(\widetilde\Omega,\Omega)$ is small enough. Then $\Omega$ is an energy subsolution (for a possibly different constant $m$). 
\end{teo}
In particular, using this result, in \cite{bulbk} and \cite{bubuve}, was proved boundedness of the optimal sets of some spectral optimization problems. In this section, we exploit the notion of a subsolution differently, obtaining an inner density estimate, which we use later in Section \ref{s5} to study the solutions of general multiphase problems.

 Lemmas \ref{sopra}  and \ref{bnd} below are implicitly contained in the paper of Alt and Caffarelli \cite[Lemma 3.4]{altcaf}. We adapt them in the context of shape subsolutions of  the  torsional energy and rephrase them in two separate statements. For the sake of completeness we report here the proofs.
\begin{lemma}\label{sopra}
Let $\Omega\subset\R^d$ be an energy subsolution with constant $m$ and let $w=w_\Omega$. Then there exist constants $C_d$, depending only on the dimension $d$, and $r_0$, depending on $\vps$, such that for each $x_0\in\R^d$ and each $0<r<r_0$ we have the following inequality:
\begin{equation}\label{sopra1}
\begin{array}{ll}
\ds \frac{1}{2}\int_{B_r(x_0)}|\nabla w|^2\,dx+m|B_r(x_0)\cap \{w>0\}|\\
\\
\qquad\qquad\qquad\le \ds \int_{B_r(x_0)}w\,dx+ C_d\left(r+\frac{\|w\|_{L^\infty(B_{2r}(x_0))}}{2r}\right)\int_{\partial B_r(x_0)}w\,d\HH^{d-1},
\end{array}
\end{equation}
\end{lemma}
\begin{proof}
Without loss of generality, we can suppose that $x_0=0$. We denote with $B_r$ the ball of radius $r$ centered in $0$ and with $A_r$ the annulus $B_{2r}\setminus \overline{B_{r}}$.\\
Let $\psi:A_1\rightarrow\R^{+}$ be the solution of the equation:
$$\Delta \psi=0,\ \hbox{on}\ A_1,\qquad \psi=0,\ \hbox{on}\ \partial B_{1},\qquad\psi=1,\ \hbox{on}\ \partial B_{2}.$$
We can also give the explicit form of $\psi$, but for our purposes, it is enough to know that $\psi$ is bounded and positive.\\
   With $\phi:A_1\rightarrow\R^{+}$ we denote the solution of the equation:
$$-\Delta \phi=1,\ \hbox{on}\ A_1,\qquad\phi=0,\ \hbox{on}\ \partial B_{1},\qquad\phi=0,\ \hbox{on}\ \partial B_{2}.$$
For an arbitrary $r>0$, $\alpha>0$ and $k>0$, we have that the solution $v$ of the equation
$$-\Delta v=1,\ \hbox{on}\ A_r,\qquad v=0,\ \hbox{on}\ \partial B_{1},\qquad v=\alpha,\ \hbox{on}\ \partial B_{2},$$
is given by 
\begin{equation}\label{sopra5}
v(x)=r^2\phi(x/r)+\alpha\psi(x/r),
\end{equation}
and it's gradient is of the form 
\begin{equation}\label{sopra6}
\nabla v(x)=r(\nabla\phi)(x/r)+\frac{\alpha}{r}\nabla\psi(x/r).
\end{equation}

Let $v$ be as in \ref{sopra5} with $\alpha\ge \|w\|_{L^{\infty}(B_{2r})}$. Consider the function $w_r=w I_{B_{2r}^c}+w\wedge v I_{B_{2r}}$ and note that, by the choice of $\alpha$, we have that $w_r\in H^1_0(\Dr)$ and denote with $\Omega_r$ the quasi-open set $\{w_r>0\}=\Omega\setminus\overline{B_{r}}$. Since $\Omega$ is an energy subsolution, choosing $r$ small enough, we have the inequality
$$\frac{1}{2}\int_{\Omega}|\nabla w|^2\,dx-\int_\Omega w(x)\,dx+m|\{w>0\}|\le \frac{1}{2}\int_{\Omega}|\nabla w_r|^2\,dx-\int_{\Omega} w_r(x)\,dx+m|\{w_r>0\}|.$$
Since $w_r=0$ in $B_{r}$ and $w_r=w$ in $(B_{2r})^c$, we have that
\begin{equation}
\begin{array}{ll}
\frac{1}{2}\int_{B_{r}}|\nabla w|^2\,dx+m|B_{r}\cap\{w>0\}|&\le \frac{1}{2}\int_{A_r}|\nabla w_r|^2-|\nabla w|^2\,dx+\int_{B_{2r}}(w-w_r)\,dx\\
\\
&\le \int_{A_r}\nabla w_r\nabla (w_r-w)\,dx+\int_{B_{2r}}(w-w_r)\,dx\\
\\
&=-\int_{A_r}\nabla v\nabla ((w-v)^+)\,dx+\int_{B_{2r}}(w-v)^+\,dx\\
\\
&=\int_{\partial B_{r}}w\frac{\partial v}{\partial n}\,d\mathcal{H}^{d-1}+\int_{B_{r}}w\,dx\\
\\
&\le \Big(r\|\nabla\phi\|_{\infty}+\frac{\alpha}{r}\|\nabla\psi\|_{\infty}\Big)\int_{\partial B_{r}}w\,d\mathcal{H}^{d-1}+\int_{B_{r}}w\,dx,
\end{array}
\end{equation}
where the last inequality is due to \eqref{sopra6}. Taking $\alpha=\|w\|_{L^\infty(B_{2r})}$, we have the claim.
\end{proof}

\begin{lemma}\label{bnd}
Let $\Omega\subset\R^d$ be an energy subsolution with constant $1$ and let $w=w_\Omega$. Then there exist constants $C_d>0$ (depending only on the dimension) and $r_0>0$ (depending on the dimension and on $\eps$ from Definition \ref{sub}) such that for every $x_0\in\R^d$ and $0<r<r_0$ the following implication holds:
\begin{equation}
\Big(\|w\|_{L^{\infty}(B_r(x_0))}\leq C_d r\Big) \Rightarrow \Big(w=0\ on\ B_{r/2}(x_0)\Big).
\label{bnd1}
\end{equation} 
\end{lemma}

\begin{proof}
Without loss of generality, we can assume that $x_0=0$. By the trace theorem for $W^{1,1}$ functions (see \cite[Theorems 3.87 and 3.88]{amfupa}), we have that 
\begin{equation}\label{bnd2}
\begin{array}{ll}
\int_{\partial B_{r/2}}w\,d\mathcal{H}^{d-1}&\le C_d\left(\frac{2}{r}\int_{B_{r/2}}w(x)\,dx+\int_{B_{r/2}}|\nabla w|\,dx\right)\\
\\
&\le C_d\left(\frac{2}{r}\int_{B_{r/2}}w(x)\,dx+\frac{1}{2}\int_{B_{r/2}}|\nabla w|^2\,dx+\frac{1}{2}|\{w>0\}\cap B_{r/2}|\right)\\
\\
&\le 2C_d\left(\frac{2}{r}\|w\|_{L^{\infty}(B_{r/2})}+\frac{1}{2}\right)\left(\frac{1}{2}\int_{B_{r/2}}|\nabla w|^2\,dx+|\{u>0\}\cap B_{r/2}|\right),
\end{array}
\end{equation}
where the constant $C_d>0$ depends only on the dimension $d$.\\
   We define the energy of $w$ on the ball $B_r$ as
\begin{equation}
E(w,B_{r})=\frac{1}{2}\int_{B_{r}}|\nabla w|^2\,dx+|B_{r}\cap \{w>0\}|.
\label{bnd3}
\end{equation}

Combining \eqref{bnd2} with the estimate from Lemma \ref{sopra1}, we have

\begin{equation}\label{bnd4}
\begin{array}{ll}
E(w,B_{r/2})\le \int_{B_{r/2}}w\,dx+C_d\left(r+\frac{2}{r}\|w\|_{L^\infty(B_{r})}\right)\int_{\partial B_{r/2}}w\,d\mathcal{H}^{d-1}\\
\\
\ \ \ \le \left(\|w\|_{L^{\infty}(B_{r/2})}+C_d\left(\frac{2}{r}\|w\|_{L^{\infty}(B_{r/2})}+\frac{1}{2}\right)\left(r+\frac{1}{r}\|w\|_{L^\infty(B_{r})}\right)\right)E(w,B_{r/2}),
\end{array}
\end{equation}
where the constants $C_d$ depend only on the dimension $d$. The claim follows by observing that if 
$$\|w\|_{L^{\infty}(B_r)}\le cr,$$
for some small $c$ and $r$, then we obtain a contradiction in \eqref{bnd4}.
\end{proof}

In other words, Lemma \ref{bnd} says that in a point of $\overline \Omega^M$  (the measure theoretic closure of the energy subsolution $\Omega$)
the function $w_\Omega$ has at least linear growth. In particular, the maximum of $w_\Omega$ on $B_r(x)$ and the average on $\partial B_r(x)$ are comparable for $r>0$ small enough.   
 
\begin{cor}\label{est}
Suppose that $\Omega\subset\R^d$ is an energy subsolution with $m=1$ and let $w=w_\Omega$. Then there exists $r_0>0$, depending on the dimension and the constant $\eps$ from Definition \ref{sub}, such that for every $x_0\in\overline\Omega^M$ and every $0<r<r_0$, we have
\begin{equation}\label{est1}
2^{-d-2}\|w\|_{L^\infty(B_r(x_0))}\le\frac{1}{|\partial B_{2r}|}\int_{\partial B_{2r}(x_0)}w\,d\HH^{d-1}\le \|w\|_{L^\infty(B_{2r}(x_0))}.
\end{equation} 
\end{cor}

\begin{proof}
Suppose that $x_0=0$. Since $w$ is positive and satisfies $\Delta w+1=0$ on $\{w>0\}=\Omega$, we have that $\Delta w+1\ge0$ on $\Dr$ (see, for example, \cite{budm93}). Consider the function 
$$\varphi_{2r}(x)=\frac{\left((2r)^d-|x|^d\right)^+}{2d},$$
solution of the equation 
$$\Delta \varphi_{2r}+1=0,\qquad \varphi_{2r}\in H^1_0(B_{2r}).$$ 
Then, we have that $\|\varphi_{2r}\|_\infty\le d^{-1}2^{d-1}r^d$ and $\Delta(w-\varphi_{2r})\ge 0$ on the ball $B_{2r}$.Thus, comparing $w-\varphi_{2r}$ with the harmonic function on $B_{2r}$ with boundary values $w$, we obtain that for every $x\in B_{r}$, we have
\begin{equation}
\begin{array}{ll}
w(x)-\varphi_{2r}(x)&\le \frac{4r^2-|x|^2}{d\omega_d 2r}\int_{\partial B_{2r}}\frac{w(y)}{|y-x|^d}\,d\HH^{d-1}(y)\\
\\
&\le \frac{2^d}{|\partial B_{2r}|}\int_{\partial B_{2r}}w\,d\HH^{d-1}.
\end{array}
\end{equation}
In particular if, for $0<r<\min\{r_0,\frac{d}{2^{d+1}}C_d,1\}$, where $r_0$ and $C_d$ are the constants from Lemma \ref{bnd}, we choose $x_r\in B_r$ such that 
$$w(x_r)>\frac{1}{2}\|w\|_{L^{\infty}(B_r)}>\frac{rC_d}{2},$$
where $C_d$ is the constant from Lemma \ref{bnd}, then we have 
\begin{equation}
\begin{array}{ll}
\frac{1}{2}\|w\|_{L^\infty(B_r)}\le w(x_r)
&\le\frac{2^d}{|\partial B_{2r}|}\int_{\partial B_{2r}}w\,d\HH^{d-1}+\frac{2^{d-1}r^d}{d}\\
\\
&\le\frac{2^d}{|\partial B_{2r}|}\int_{\partial B_{2r}}w\,d\HH^{d-1}+\frac{2^{d-1}r^{d-1}}{d}\frac{d}{2^{d+1}}C_d\\
\\
&\le\frac{2^d}{|\partial B_{2r}|}\int_{\partial B_{2r}}w\,d\HH^{d-1}+\frac{
r^{d-2}}{4}\|w\|_{L^\infty(B_r)},
\end{array}
\end{equation}
and so, the claim.
\end{proof}

\begin{oss}
In particular, there are constants $c$ and $r_0$ such that if $x_0\in\overline\Omega^M$, then for every $0<r\le r_0$, we have that
$$cr\le \frac{1}{|\partial B_r|}\int_{\partial B_r(x_0)}w_\Omega\,d\HH^{d-1}.$$
Moreover, since $\int_{B_r}w_\Omega\,dx=\int_0^r\int_{\partial B_s}w_\Omega\,d\HH^{d-1}\,ds$, we have
$$cr\le \frac{1}{|B_r|}\int_{B_r(x_0)}w_\Omega\,dx.$$
\end{oss}

As a consequence of Corollary \ref{est}, we can simplify \eqref{sopra1}. Precisely, we have the following result.

\begin{cor}\label{enest}
Suppose that $\Omega\subset\R^d$ is an energy subsolution with $m=1$ and let $w:=w_\Omega$. Then there are constants $C_d>0$, depending only on the dimension $d$, and $r_0$, depending on $d$ and $\eps$ from Definition \ref{sub}, such that for every $x_0\in\overline\Omega^M$ and $0<r<r_0$, we have
\begin{equation}\label{enest1}
\frac{1}{2}\int_{B_r(x_0)}|\nabla w|^2\,dx+|\{w>0\}\cap B_r(x_0)|\le C_d\frac{\|w\|_{L^\infty(B_{2r}(x_0))}}{2r}\int_{\partial B_r(x_0)}w\,d\HH^{d-1}.
\end{equation}
\end{cor}

\begin{proof}
By Lemma \ref{bnd} and Corollary \ref{est}, for $r>0$ small enough, we have
\begin{equation}\label{enest2}
\frac{1}{r}\|w\|_{L^\infty(B_r(x_0))}\ge C_d,\qquad\frac{1}{r|\partial B_r|}\int_{\partial B_{r}(x_0)}w\,d\HH^{d-1}\ge 2^{-d-2}C_d .
\end{equation}
Thus, for $r$ as above, we have 
\begin{equation}
\int_{B_r(x_0)}w(x)\,dx\le |B_r|\frac{d2^{-d-2}C_d}{r}\|w\|_{L^\infty(B_r(x_0))}\le \frac{1}{r}\|w\|_{L^\infty(B_r(x_0))}\int_{\partial B_r(x_0)}w\,d\HH^{d-1}, 
\end{equation}
and so, it remains to apply the above estimate to \eqref{sopra1}.
\end{proof}

Relying on inequality \eqref{enest1} and Lemma \ref{bnd} we get the following inner density estimate, which is much weaker than the density estimates from \cite{altcaf}. The main reason is that we work only with subsolutions and not with minimizers of a free boundary problem. 

\begin{prop}\label{dens}
Suppose that $\Omega\subset\R^d$ is an energy subsolution and let $w=w_\Omega$. Then there exists a constant $c>0$, depending only on the dimension, such that for every $x_0\in\overline\Omega^M$, we have
\begin{equation}\label{dens0}
\limsup_{r\to 0}\frac{|\{w>0\}\cap B_{r}(x_0)|}{|B_{r}|}\ge c.
\end{equation}

%
\end{prop}

\begin{proof}
Without loss of generality, we can suppose that $x_0=0$ and by rescaling we can assume that $m=1$. Let $r_0$ and $C_d$ be as in Lemma \ref{bnd} and let $0<r<r_0$. By the Trace Theorem in $W^{1,1}(B_r)$, we have
\begin{equation}\label{dens1}
\begin{array}{ll}
\int_{\partial B_r}w\,d\HH^{d-1}&\le C_d\left(\int_{B_r}|\nabla w|\,dx+\frac{1}{r}\int_{B_r}w\,dx\right)\\
\\
&\le C_d\left(\left(\int_{B_r}|\nabla w|^2\,dx\right)^{1/2}|\{w>0\}\cap B_r|^{1/2}+\frac{\|w\|_{L^\infty(B_r)}}{r}|\{w>0\}\cap B_r|\right)\\
\\
&\le C_d\left(\frac{\|w\|_{L^\infty(B_{2r})}}{2r}\int_{\partial B_r}w\,d\HH^{d-1}\right)^{1/2}|\{w>0\}\cap B_r|^{1/2}\\
\\
&\qquad\qquad+C_d\frac{\|w\|_{L^\infty(B_r)}}{r}|\{w>0\}\cap B_r|,
\end{array}
\end{equation}
where the last inequality is 
due to Corollary \ref{enest} and $C_d$ denotes a constant which depends only on the dimension $d$. Let 
\begin{equation}
\begin{array}{lcl}
X&=&\left(\int_{\partial B_r}w\,d\HH^{d-1}\right)^{1/2},\\
\\
\alpha&=&C_d\left(\frac{\|w\|_{L^\infty(B_{2r})}}{2r}\right)^{1/2}|\{w>0\}\cap B_r|^{1/2},\\
\\
\beta&=&C_d\frac{\|w\|_{L^\infty(B_r)}}{r}|\{w>0\}\cap B_r|.
\end{array}
\end{equation}
Then, we can rewrite \eqref{dens1} as 
$$X^2\le \alpha X+\beta.$$
But then, since $\alpha,\beta>0$, we have the estimate $X\le\alpha+\sqrt\beta$. Taking the square of both sides, we obtain
\begin{equation}\label{dens2}
\begin{array}{ll}
\int_{\partial B_r}w\,d\HH^{d-1}&\le C_d |\{w>0\}\cap B_r| \left(\frac{\|w\|_{L^\infty(B_{2r})}}{2r}+\frac{\|w\|_{L^\infty(B_{r})}}{r}\right)\\
\\
&\le 3C_d |\{w>0\}\cap B_r|\frac{\|w\|_{L^\infty(B_{2r})}}{2r}.
\end{array}
\end{equation}
By Corollary \ref{est}, we have that
\begin{equation}\label{dens3}
\frac{\|w\|_{L^\infty(B_{r/2})}}{r/2}\le \frac{C_d |\{w>0\}\cap B_r|}{|B_r|}\frac{\|w\|_{L^\infty(B_{2r})}}{2r},
\end{equation} 
for some constant $C_d$ depending only on the dimension. We choose the constant $c$ from \eqref{dens0} as $c=(2C_d)^{-1}$ and we will prove that \eqref{dens0} does not hold. Suppose, by absurd, that we have 
\begin{equation}\label{dens4}
\limsup_{r\to 0}\ C_d\frac{|\{w>0\}\cap B_{r}|}{|B_{r}|}<\frac{1}{2}.
\end{equation}
Setting, for $r>0$ small enough,   
$$f(r):=\frac{\|w\|_{L^\infty(B_{r})}}{r},$$
and using \eqref{dens3}, we have that for each $n\in\N$ the following inequality holds 
\begin{equation}\label{dens5}
f(r4^{-(n+1)})\le \frac{C_d|\{w>0\}\cap B_{2r4^{-(n+1)}}|}{|B_{2r4^{-(n+1)}}|}f(r4^{-n}),
\end{equation}
and so
\begin{equation}\label{p5}
f(r4^{-(n+1)})\le f(r)\prod_{k=0}^n\frac{C_d|\{w>0\}\cap B_{2r4^{-(k+1)}}|}{|B_{2r4^{-(k+1)}}|}.
\end{equation}

By equation \eqref{dens4}, we have that $f(r4^{-n})\to 0$, which is a contradiction with Lemma \ref{bnd}.
\end{proof}

\begin{teo}\label{thss}
Suppose that the quasi-open set $\Omega\subset\R^d$ is an energy subsolution with constant $m>0$. Then, we have that:
\begin{enumerate}[(i)]
\item $\Omega$ is a bounded set;
\item $\Omega$ is of finite perimeter and 
\begin{equation}
\sqrt{\frac{m}{2}}{\mathcal H}^{d-1} (\partial ^*\Omega)\le|\Omega|;
\end{equation} 
\item $\Omega$ is equivalent a.e. to a closed set. More precisely, $\Omega=\overline{\Omega}^M$ a.e., $\overline\Omega^M=\R^d\setminus\Omega_{(0)}$ and $\Omega_{(0)}$ is an open set. Moreover, if $\Omega$ is given through its canonical representative from Remark \ref{canrepqo}, then $\overline\Omega=\overline\Omega^M$. 
\end{enumerate}
\end{teo}
\begin{proof}
The first two statements    concerning the boundedness and the perimeter of $\Omega$ were implicitly proved in \cite[Theorem 2.2]{bulbk}. For the third one it is sufficient to prove that $\Omega_{(0)}$ satisfies
\begin{equation}\label{openss1303013}
\Omega_{(0)}=\R^d\setminus\overline\Omega^M=\left\{x\in\R^d:\ \hbox{exists}\ r>0\ \hbox{such that}\ |B_r(x)\cap\Omega|=0\right\},
\end{equation}
where the second equality is just the definition of $\overline\Omega^M$. We note that $\Omega_{(0)}\subset\R^d\setminus\overline\Omega^M$ trivially holds for every measurable $\Omega$. On the other hand, if $x\in\overline\Omega^M$, then, by Proposition \ref{dens}, there is a sequence $r_n\to0$ such that 
$$\lim_{n\to\infty}\frac{|B_{r_n}(x)\cap\Omega|}{|B_{r_n}|}\ge c>0,$$
and so $x\notin\Omega_{(0)}$, which proves the opposite inclusion and the equality in \eqref{openss1303013}. 
\end{proof}

\begin{oss}
The second statement of Theorem \ref{thss} implies, in particular, that the energy subsolutions cannot be too small. Indeed, by the isoperimetric inequality, we have 
$$c_d\sqrt{\frac{m}{2}}|\Omega|^{\frac{d-1}{d}}\le \sqrt{\frac{m}{2}}{\mathcal H}^{d-1} (\partial ^*\Omega)\le |\Omega|\le C_d [{\mathcal H}^{d-1} (\partial ^*\Omega)]^{\frac{d}{d-1}},$$
and so
$$c_dm^{\frac{d}{2}}\le |\Omega|,\qquad c_dm^{\frac{d-1}{2}}\le {\mathcal H}^{d-1} (\partial ^*\Omega),$$
for some dimensional constant $c_d$. 
\end{oss}

The results of this section can be adapted to the subsolutions for first Dirichlet eigenvalue, i.e. the quasi-open sets $\Omega\subset\R^d$ such that there are real constants $m>0$ and $\varepsilon>0$ such that for each quasi-open set $\tilde\Omega\subset\Omega$, for which $d_\g(\widetilde\Omega,\Omega)<\varepsilon$, we have 
\begin{equation}\label{sub1.1}
\lambda_1(\Omega)+m|\Omega|\le \lambda_1(\tilde\Omega)+m|\tilde\Omega|.
\end{equation}
Subsolutions for the first Dirichlet eigenvalue are also subsolutions for the energy, so Theorem  \ref{gammalipth} applies. Moreover, we have the following new, or more precise, statements.
\begin{teo}\label{thsslb1}
Suppose that the quasi-open set $\Omega\subset\R^d$ is a subsolution  for the first eigenvalue of the Dirichlet Laplacian. Then, we have that:
\begin{enumerate}[(i)]
\item 
\begin{equation}\label{thsslb1e1}
\sqrt{m}{\mathcal H}^{d-1} (\partial ^*\Omega)\le\lambda_1(\Omega)|\Omega|^{1/2};
\end{equation} 
\item $\Omega$ is quasi-connected, i.e. if $A,B\subset\Omega$ are two quasi-open sets such that $A\cup B=\Omega$ and $\cp(A\cap B)=0$, then $\cp(A)=0$ or $\cp(B)=0$;
\item $\Omega=\{u>0\}$, up to a set of zero capacity, where $u$ is the first Dirichlet eigenfunction on $\Omega$.
\end{enumerate}
\end{teo}
     
\begin{proof}
 In order to prove the bound \eqref{thsslb1e1}, we follow the idea from \cite{bulbk}.
Let $u$ be the first, normalized in $L^2(\Omega)$, eigenfunction on $\Omega$. Since $\lambda_1(\{u>0\})=\lambda_1(\Omega)$, we have that $|\{u>0\}\Delta\Omega |=0$. Consider the set $\Omega_{\eps}=\{u>\eps\}$. In order to use $\Omega_\eps$ to test the (local) subminimality of $\Omega$, we first note that $\Omega_\eps$ $\gamma$-converges to $\Omega$. Indeed, the family of torsion functions $w_\eps$ of $\Omega_\eps$ is decreasing in $\eps$ and converges in $L^2$ to the torsion function $w$ of $\{u>0\}$, as $\eps\to 0$, since
$$\lambda_1(\Omega)\int_\Omega (w-w_\eps)u\,dx=\int_\Omega\nabla w\cdot\nabla u\,dx-\int_{\Omega_\eps}\nabla w_\eps\cdot\nabla (u-\eps)^+\,dx=\int_\Omega u-(u-\eps)^+\,dx\to 0.$$
Now, using $(u-\eps)^+\in H^1_0(\Omega_\eps)$ as a test function for $\lambda_1(\Omega_\eps)$, we have
\be
\begin{array}{ll}
\lambda_1(\Omega)+m|\Omega|&\ds\le \lambda_1(\Omega_{\eps})+m|\Omega_{\eps}|\\
\\
&\ds\le\frac{\int_\Omega|\nabla (u-\eps)^{+}|^2\,dx}{\int_\Omega|(u-\eps)^{+}|^2\,dx}+m|\Omega_{\eps}|\\
\\
&\ds\le\int_\Omega|\nabla (u-\eps)^{+}|^2\,dx+\lambda_1(\Omega_\eps)\int_\Omega \left(|(u-\eps)^{+}|^2-u^2\right)dx+m|\Omega_{\eps}|\\
\\
&\ds\le\int_\Omega|\nabla (u-\eps)^{+}|^2\,dx+\lambda_1(\Omega)2\eps\int_\Omega u\,dx+m|\Omega_{\eps}|\\
\\
&\ds\le\int_\Omega|\nabla (u-\eps)^{+}|^2\,dx+2\eps\lambda_1(\Omega)|\Omega|^{1/2}+m|\Omega_{\eps}|.
\end{array}
\ee
Thus, we obtain
\begin{equation}
\int_{\{0<u\le\eps\}}|\nabla u|^2\,dx+m|\{0<u\le\eps\}|\le 2\eps\lambda_1(\Omega)|\Omega|^{1/2}.
\end{equation}
The mean quadratic-mean geometric and the H\"older inequalities give
\begin{equation}
2m^{1/2}\int_{\{0<u\le\eps\}}|\nabla u|\,dx\le 2m^{1/2}\left(\int_{\{0<u\le\eps\}}|\nabla u|^2\,dx\right)^{1/2}|\{0<u\le\eps\}|^{1/2}\le 2\eps\lambda_1(\Omega)|\Omega|^{1/2}.
\end{equation}
Using the co-area formula, we obtain
\begin{equation}
\frac{1}{\eps}\int_0^\eps \HH^{d-1}(\{u>t\}^*)\,dt\le m^{-1/2}\lambda_1(\Omega)|\Omega|^{1/2},
\end{equation}
and so, passing to the limit as $\eps\to0$, we obtain \eqref{thsslb1e1}.

   Let us now prove $(ii)$. Suppose, by absurd that $\cp(A)>0$ and $\cp(B)>0$ and, in particular, $|A|>0$ and $|B|>0$. Since $\cp(A\cap B)=0$, we have that $H^1_0(\Omega)=H^1_0(A)\oplus H^1_0(B)$ and so, $\lambda_1(\Omega)=\min\{\lambda_1(A),\lambda_1(B)\}$. Without loss of generality, we may suppose that $\lambda_1(\Omega)=\lambda_1(A)$. Then, we have
$$\lambda_1(A)+m|A|<\lambda_1(A)+m (|A|+|B|)=\lambda_1(\Omega)+m|\Omega|,$$
which is a contradiction with the subminimality of $\Omega$.

  In order to see $(iii)$, it is sufficient to prove that for every quasi-connected $\Omega$, we have $\Omega=\{u>0\}$. Indeed, let $\omega=\{u>0\}$ and consider the torsion functions $w_\omega$ and $w_\O$. We note that, by the weak maximum principle, we have $w_\omega\le w_\O$. Setting $\lambda=\lambda_1(\Omega)$, we have
$$\int_\O \lambda u w_\omega\,dx=\int_\O \nabla u\cdot\nabla w_\omega\,dx=\int_\O u\,dx,$$
$$\int_\O \lambda u w_\Omega\,dx=\int_\O \nabla u\cdot\nabla w_\Omega\,dx=\int_\O u\,dx.$$
Subtracting, we have
\begin{equation}\label{lemmaquasicon1}
\int_\O u(w_\O-w_\omega)\,dx=0,
\end{equation} 
and so, $w_\O=w_\omega$ on $\omega$. Consider the sets $A=\O\cap\{w_\Omega=w_\omega\}$ and $B=\O\cap\{w_\Omega>w_\omega\}$. By construction, we have that $A\cup B=\O$ and $A\cap B=\emptyset$. Moreover, we observe that $A=\omega\neq\emptyset$. Indeed, one inclusion $\omega\subset A$, follows by \eqref{lemmaquasicon1}, while the other inclusion follows, since by strong maximum principle for $w_\omega$ and $w_\Omega$ we have the equality $$\O\cap\{w_\Omega=w_\omega\}=\{w_\Omega>0\}\cap\{w_\Omega=w_\omega\}\subset\{w_\omega>0\}=\omega.$$
By the quasi-connectedness of $\Omega$, we have that $B=\emptyset$, i.e. $\omega=\Omega$.   
   
\end{proof}

\begin{oss}
If $\Omega$ is a subsolution for the first Dirichlet eigenvalue \eqref{sub1.1}, then we have the following bound on $\lambda_1(\Omega)$:  
\begin{equation}\label{boundlamsublb1}
\lambda_1(\Omega)\ge c_d m^{\frac{2}{d+2}},
\end{equation}
where $c_d$ is a dimensional constant. In fact, by \eqref{thsslb1e1} and the isoperimetric inequality, we have
$$\lambda_1(\Omega)|\Omega|^{1/2}\ge\sqrt{m}P(\Omega)\ge c_d\sqrt m|\Omega|^{\frac{d-1}{d}},$$
and so
$$\lambda_1(\Omega)\ge c_d\sqrt{m}|\Omega|^{\frac{d-2}{2d}}.$$
By the Faber-Krahn inequality $\lambda_1(\Omega)|\Omega|^{2/d}\ge \lambda_1(B)|B|^{2/d}$, we obtain
$$\lambda_1(\Omega)\ge c_d\sqrt{m}\left(|\Omega|^{\frac{2}{d}}\right)^{\frac{d-2}{4}}\ge c_d\sqrt{m}\left(\lambda_1(\Omega)^{-1}\lambda_1(B)|B|^{2/d}\right)^{\frac{d-2}{4}}\ge c_d\sqrt{m}\lambda_1(\Omega)^{-\frac{d-2}{4}}.$$
\end{oss}

\begin{oss}\label{buveqo1}
Even if the subsolutions have some nice qualitative properties, their local behaviour might be very irregular. In fact, one may construct subsolutions for the first Dirichlet eigenvalue (and thus, energy subsolutions) with empty interior in sense of the Lebesgue measure, i.e. the set $\Om_{(1)}$ of points of density $1$ has empty interior. Consider a bounded quasi-open set $\Dr$ with empty interior as, for example, 
$$\Dr=(0,1)\times(0,1)\setminus\left(\bigcup_{i=1}^{\infty}\overline B_{r_i}(x_i)\right)\subset\R^2,$$
where $\{x_i\}_{i\in\N}=\Q$ and $r_i$ is such that
$$\sum_{i\in \N} \cp (\overline B_{r_i}(x_i)) < +\infty \;\;\mbox{and}\; \sum_{i\in \N} \pi r_i^2 < \frac12.$$
 Let $\Omega\subset\Dr$ be the solution of the problem
$$\min\left\{\lambda_1(\Omega)+|\Omega|:\Omega\subset\Dr,\ \Omega\ \hbox{quasi-open}\right\}.$$
Since, $\Omega$ is a global minimizer among all sets in $\Dr$, it is also a subsolution. On the other hand, $\Dr$ has empty interior and so does $\Omega$. 
\end{oss}

\section{Interaction between energy subsolutions}\label{s4}
   In this section we consider configurations of disjoint quasi-open sets $\Omega_1,\dots,\Omega_n$ in $\R^d$, each one being an energy subsolution. In particular, we will study the behaviour of the energy functions $w_{\Omega_i}$, $i=1,\dots,n$, around the points belonging to more than one of the measure theoretical boundaries $\partial^M\Omega_i$. 
      
   We start our discussion with a  result which is useful in  multiphase shape optimization problems, since it allows to separate by an open set each quasi-open cell from the others.

\begin{lemma}\label{sep2qolemma}
Suppose that the disjoint quasi-open sets $\Omega_1$ and $\Omega_2$ are energy subsolutions. Then the corresponding energy function $w_1$ and $w_2$ vanish quasi-everywhere (and so, also a.e.) on the common boundary $\partial^M\Omega_1\cap\partial^M\Omega_2$.  
\end{lemma}
\begin{proof}
By Remark \ref{canrepqo} we may suppose that $\Omega_i=\{w_i>0\}$ and that every point $x\in\R^d$ is a regular point for both $w_1$ and $w_2$.

Let $x_0\in\R^d$ be such that $w_2(x_0)>0$. Suppose, by absurd, that $x_0\in \partial^M\Omega_1\cap\partial^M\Omega_2$. In particular, for each $r>0$ we have $|\{w_1>0\}\cap B_r(x_0)|>0$. By Proposition \ref{dens}, we have that
there is a sequence $r_n\to 0$ such that 
\begin{equation}\label{sep2qo2}
\lim_{n\to\infty}\frac{|\{w_1>0\}\cap B_{r_n}(x_0)|}{|B_{r_n}|}\ge c>0.
\end{equation}
Since $|\Omega_1\cap\Omega_2|=0$, we have that 
\begin{equation}\label{sep2qo3}
\limsup_{n\to\infty}\frac{|\{w_2>0\}\cap B_{r_n}(x_0)|}{|B_{r_n}|}\le 1-c<1,
\end{equation}
and since $x_0$ is a regularity point for $w_2$, we obtain
\begin{align*}
w_2(x_0)&=\lim_{n\to\infty}\frac{1}{|B_{r_n}|}\int_{B_{r_n}(x_0)}w_2(x)\,d x\\
\\
&\le \limsup_{n\to\infty}\|w_2\|_{L^\infty(B_{r_n}(x_0))} \limsup_{n\to\infty}\frac{|\{w_2>0\}\cap B_{r_n}(x_0)|}{|B_{r_n}|}\label{sep2qo4}\\
\\
&\le (1-c)\limsup_{n\to\infty}\|w_2\|_{L^\infty(B_{r_n}(x_0))}.
\end{align*}
Note that, in order to have a contradiction, it is enough to prove that 
$$\lim_{r\to 0}\|w_2\|_{L^\infty(B_r(x))}=w_2(x_0).$$
In fact, suppose that there is a sequence $x_n\to x_0$ such that $w_2(x_n)\ge\delta+w_2(x_0)$ for some $\delta\ge0$. Let $r>0$ and let 
$$v_n(x)=w_2(x)-\frac{r^2-|x_n-x|^2}{2d}.$$
Then $\Delta v_n\ge 0$ on $B_r(x_n)$, and so 
$$v_n(x_n)\le \frac{1}{|B_r|}\int_{B_r(x_n)}v_n\,dx\le \frac{1}{|B_r|}\int_{B_r(x_n)}w_2\,dx.$$
By the choice of $v_n$, we have that 
$$w_2(x_n)\le \frac{r^2}{2d}+\frac{1}{|B_r|}\int_{B_r(x_n)}w_2\,dx,$$
and since the map $x\mapsto\int_{B_r(x)}w_2\,dx$ is continuous, we obtain
$$\delta+w_2(x_0)\le \frac{r^2}{2d}+\frac{1}{|B_r|}\int_{B_r(x_0)}w_2\,dx.$$
Passing to the limit as $r\to0$, we have that $\delta=0$. As a consequence, we have that $w_2(x_0)=0$, which is a contradiction.
\end{proof}

\begin{prop}\label{sep2qo}
Suppose that the disjoint quasi-open sets $\Omega_1$ and $\Omega_2$ are energy subsolutions. Then there are open sets $D_1, D_2\subset\R^d$ such that $\Omega_1\subset D_1$, $\Omega_2\subset D_2$ and $\Omega_1\cap D_2=\Omega_2\cap D_1=\emptyset$, up to sets of zero capacity.
\end{prop}
\begin{proof}
Define  $D_1=\R^d\setminus\overline{\Omega}_2^M$ and $D_2=\R^d\setminus\overline{\Omega}_1^M$, which by the definition of a measure theoretic closure are open sets. As in Lemma \ref{sep2qolemma}, we may suppose that $\Omega_i=\{w_i>0\}$ and that every point of $\Omega_i$ is regular for $w_i$. Since $\Omega_i\subset\overline{\Omega}_i^M$, we have to show only that $\Omega_1\cap\overline{\Omega}_2^M=\emptyset$. Indeed, if this is not the case there is a point $x_0\in\overline{\Omega}_2^M$ such that $w(x_1)>0$, which is a contradiction with Lemma \ref{sep2qolemma}.
\end{proof}

   Let now $\Omega$ be an energy subsolution and let $w=w_\Omega$. Then by Lemma \ref{bnd} there is a constant $c>0$ such that for any $x_0\in \partial^M\Omega$ and $r>0$ small enough we have that, 
\begin{equation}\label{twosup130308}
cr\le\|w\|_{L^\infty(B_r(x_0))}.
\end{equation} 
\begin{oss}\label{triple130308}
{\it (Intuitive approach in the smooth case)} Note that in a sufficiently smooth setting,  \eqref{twosup130308} corresponds in some weak sense to a lower bound on the gradient of $w$ in $x_0$, i.e. it is an alternative (not equivalent!) form of the inequality
\begin{equation}\label{twograd130308}
c\le\frac{1}{r^d}\int_{B_r(x_0)}|\nabla w|^2\,dx.
\end{equation} 
The later can be used to determine some quantitative behaviour of the some optimal partitions.

Let $\Omega_1$ and $\Omega_2$ be two disjoint quasi-open sets and $x_0\in\partial^M\Omega_1\cap\partial^M\Omega_2$. If, for $r$ small enough, \eqref{twograd130308} holds for $w_i:=w_{\Omega_i}$ and $i=1,2$, then applying the Caffarelli-Jerison-Kenig monotonicity formula (Theorem \ref{mth2}), there is a constant $C>0$ such that for $r$ small enough we have 
\begin{equation*}
\frac{1}{r^d}\int_{B_r(x_0)}|\nabla w_i|^2\,dx\le C,
\end{equation*}
i.e. the gradients $|\nabla w_1|$ and $|\nabla w_2|$ are bounded in $\partial^M\Omega_1\cap\partial^M\Omega_2$.

Let $\Omega_1$, $\Omega_2$ and $\Omega_3$ be three disjoint quasi-open sets such that, for each $x_0\in\partial^M\Omega_i$ and $i=1,2,3$, the corresponding torsion functions $w_i$ satisfies \eqref{twograd130308}. Then the set $\partial^M\Omega_1\cap\partial^M\Omega_2\cap\partial^M\Omega_3$ has to be empty. Indeed, if $x_0\in \partial^M\Omega_1\cap\partial^M\Omega_2\cap\partial^M\Omega_3$, by the three-phase monotonicity formula (Lemma \ref{mth3}) and \eqref{twograd130308} we would have 
\begin{equation*}
r^{-3\eps}c^3\le \prod_{i=1}^3\left(\frac{1}{r^{d+\eps}}\int_{B_r(x_0)}|\nabla w_i|^2\,dx\right)\le C_d\left(1+\sum_{i=1}^3\int_{B_1(x_0)}\frac{|\nabla w_i|^2}{|x|^{d-2}}\,dx\right),
\end{equation*}   
which is false for $r>0$ small enough. So, triple junction points can not exist. 
\end{oss}
\medskip

\begin{oss} {\it (The two dimensional case)}
In dimension two, inequality \eqref{twograd130308} does not require smoothness, being an easy consequence from the Sobolev inequality. Indeed, let $\Omega_1,\Omega_2\subset\R^2$ be two disjoint energy subsolution with $m=1$ and let $x_0\in\partial^M\Omega_1\cap\partial^M\Omega_2$. There is some constant $c>0$ (not depending on $\Omega_1$ and $\Omega_2$) such that \eqref{twograd130308} holds for $r>0$ small enough. Suppose $x_0=0$. By Corollary \ref{est}, for each $0<r\le r_0$, we have 
\begin{equation}\label{twosubs130306}
cr\le\frac{1}{|\partial B_r|}\int_{\partial B_r} w_1\,d\HH^{1}\qquad\hbox{and}\qquad cr\le\frac{1}{|\partial B_r|}\int_{\partial B_r} w_2\,d\HH^{1},
\end{equation} 
and, in particular, $\partial B_r\cap\{w_1=0\}\neq\emptyset$ and $\partial B_r\cap\{w_2=0\}\neq\emptyset$. Thus, for almost every $r\in(0,r_0)$, we have 
$$cr^3\le \frac{1}{|\partial B_r|}\left(\int_{\partial B_r} w_i\,d\HH^1\right)^2\le\int_{\partial B_r} w_i^2\,d\HH^1\le \lambda r^2\int_{\partial B_r}|\nabla w_i|^2\,d\HH^1,$$
where $\lambda<+\infty$ a constant. 
Dividing by $r^2$ and integrating for $r\in[0,R]$, where $R<r_0$, we obtain \eqref{twograd130308}. 

   In particular, if $\Omega_1,\Omega_2,\Omega_3\subset\R^2$ are three disjoint energy subsolutions then there are no triple points, i.e. the set $\partial^M\Omega_1\cap\partial^M\Omega_2\cap\partial^M\Omega_3$ is empty. 
\end{oss}
 \medskip
 
 In the rest of the section we make the previous arguments rigorous in the non-smooth setting and  prove that if the quasi-open sets $\Omega_1$, $\Omega_2$ and $\Omega_3$ are energy subsolutions, then the above conclusions still hold in any dimension of the space.
 We state a preliminary lemma, which is implicitly contained in the proof of Lemma 3.2 of \cite{altcaf}.
 
\begin{lemma}
For every $u\in H^1(B_r)$ we have the following estimate:
\begin{equation}\label{caf}
\frac{1}{r^2}|\{u=0\}\cap B_r|\left(\frac{1}{|\partial B_r|}\int_{\partial B_r}u\,d\HH^{d-1}\right)^2\leq C_d\int_{B_r}|\nabla u|^2\,dx,
\end{equation}
where $C_d$ is a constant that depends only on the dimension $d$.
\end{lemma}

\begin{proof}
We note that it is sufficient to prove the result in the case $u\ge 0$. Let $v\in H^1(B_r)$ be the solution of the obstacle problem
\begin{equation*}
\min\left\{\int_{B_r}|\nabla v|^2\,dx:\ u-v\in H^1_0(B_r),\ v\ge u\right\}.
\end{equation*} 
Then $v$ is super-harmonic on $B_r$ and harmonic on the quasi-open set $\{v>u\}$. Reasoning as in \cite[Lemma 2.3]{altcaf}, we have 

\begin{equation}\label{caf}
\frac{1}{r^2}|\{u=0\}\cap B_r|\left(\frac{1}{|\partial B_r|}\int_{\partial B_r}u\,d\HH^{d-1}\right)^2\leq C_d\int_{B_r}|\nabla (u-v)|^2\,dx.
\end{equation}
Now the claim follows by the harmonicity of $v$ on $\{v>u\}$ and the calculation
$$\int_{B_r}|\nabla (u-v)|^2\,dx=\int_{B_r}|\nabla u|^2-|\nabla v|^2\,dx+2\int_{B_r}\nabla v\cdot\nabla (v-u)\,dx\le\int_{B_r}|\nabla u|^2\,dx.$$
\end{proof}

\begin{teo}\label{lem130306}
Suppose that $\Omega_1,\Omega_2,\Omega_3\subset\R^d$ are three mutually disjoint energy subsolutions. Then the set $\partial^M\Omega_1\cap\partial^M\Omega_2\cap\partial^M\Omega_3$ is empty. 
\end{teo}
\begin{proof} 
Suppose for contradiction that there is a point $x_0\in\partial^M\Omega_1\cap\partial^M\Omega_2\cap\partial^M\Omega_3$. Without loss of generality $x_0=0$. Using the inequality \eqref{dens3}, we have 
\begin{equation*}
\prod_{i=1}^3\frac{\|w_i\|_{L^\infty(B_{r/2})}}{r/2}\le C_d\left(\prod_{i=1}^3\frac{|\{w_i>0\}\cap B_r|}{|B_r|}\right)\left(\prod_{i=1}^3\frac{\|w_i\|_{L^\infty(B_{2r})}}{2r}\right),
\end{equation*} 
and reasoning as in Proposition \ref{dens}, we obtain that there is a constant $c>0$ and a decreasing sequence of positive real numbers $r_n\to0$ such that
\begin{equation*}
c\le \prod_{i=1}^3\frac{|\{w_i>0\}\cap B_{r_n}|}{|B_{r_n}|},\ \forall n\in\N,
\end{equation*}
and so, for each $i=1,2,3$, we have
\begin{equation*}
c\le \frac{|\{w_i>0\}\cap B_{r_n}|}{|B_{r_n}|},\ \forall n\in\N.
\end{equation*}
Using Lemma \ref{bnd}, Corollary \ref{est} and Lemma \ref{caf} for $r=r_n$, we obtain
\begin{equation}\label{caf}
c\le \frac{|\{w_i=0\}\cap B_{r_n}|}{|B_{r_n}|}\left(\frac{1}{r_n|\partial B_{r_n}|}\int_{\partial B_{r_n}}u\,d\HH^{d-1}\right)^2\leq\frac{C_d}{r_n^d}\int_{B_r}|\nabla w_i|^2\,dx,
\end{equation}
which proves that \eqref{twograd130308} holds for every $i=1,2,3$. Now the conclusion follows as in Remark \ref{triple130308}.
\end{proof}

\begin{oss}
Let $\Omega_1,\dots,\Omega_h\subset\R^d$ be a family of disjoint energy subsolutions. Then we can classify the points in $\R^d$ in three groups, as follows:
\begin{itemize}
\item Simple points 
\begin{equation*}
Z_1=\left\{x\in\R^d:\ \exists\Omega_i>0\ \hbox{s.t.}\ x\notin\partial^M\Omega_j,\ \forall j\neq i\right\}.
\end{equation*}

\item Internal double points
$$Z_2^i=\left\{x\in\R^d:\ \exists i\neq j\ \hbox{s.t.}\ x\in \partial^M\Omega_i\cap\partial^M\Omega_j;\ \exists r>0\ \hbox{s.t.}\ |B_r(x)\cap(\Omega_i\cup\Omega_j)^c|=0\right\}.$$
\item Boundary double points 
$$Z_2^b=\left\{x\in\R^d:\ \exists i\neq j\ \hbox{s.t.}\ x\in \partial^M\Omega_i\cap\partial^M\Omega_j;\ |B_r(x)\cap(\Omega_i\cup\Omega_j)^c|>0, \ \forall r>0\right\}.$$
\end{itemize} 
\end{oss}

\section{Multiphase shape optimization problems}\label{s5}
Let $\Dr\subset\R^d$ be a bounded open set. In this section we consider shape optimization problems of the form 
\begin{equation}\label{opft1e1}
\min\Big\{g\left(F_1(\Omega_1),\dots,F_h(\Omega_h)\right)+m\sum_{i=1}^h|\Omega_i|:\ \Omega_i\in\mathcal{A}_{\cp}(\Dr),\ \forall i;\ \Omega_i\cap\Omega_j=\emptyset,\ \forall i\neq j\Big\},
\end{equation}
where $g:\R^h\to\R$ is increasing in each variable and l.s.c., $F_1,\dots,F_h: \mathcal{A}_{\cp}(\Dr) \ra \R$ are decreasing with respect to inclusions and continous for the $\g$-convergence, and $m \ge 0$ is a given constant.  Problem \eqref{opft1e1} admits a solution following Theorem  \ref{introth130312}.

\begin{deff}\label{gammalip}
We say that  $F:\mathcal{A}_{\cp}(\Dr)\to\R$ is locally $\g$-Lipschitz for sub domains (or simply $\gamma$-Lip), if for each  $\Omega \in \mathcal{A}_{\cp}(\Dr)$, there are constants $C>0$ and $\eps>0$ such that 
\begin{equation*}
|F(\widetilde\Omega)-F(\Omega)|\le Cd_\g(\widetilde\Omega,\Omega),
\end{equation*}
for every quasi-open set $\widetilde\Omega\subset\Omega$, such that  $d_\g(\widetilde\Omega,\Omega)\le\eps$.
\end{deff}

\begin{oss}
Following Theorem \ref{gammalipth}, we have that the functional associated to the $k$-th eigenvalue of the Dirichlet Laplacian $\Omega\mapsto\lambda_k(\Omega)$ is $\gamma$-Lip, for every $k\in\N$.
\end{oss}
    
\begin{teo}\label{opft2}
Assume that  $g$ is locally Lipschitz continuous,  each of the functionals $F_i$, $i=1,\dots,h$ is $\gamma$-Lip and  $m>0$ and $(\Omega_1,\dots,\Omega_h)$ is a solution of \eqref{opft1e1}. Then every quasi-open set $\Omega_i$, $i=1,\dots,h$, is an energy subsolution.
\end{teo}   

\begin{proof}
Let $\widetilde\Omega_1\subset\Omega_1$ be a quasi-open set such that $d_\g(\widetilde\Omega_1,\Omega_1)<\eps$. By the Lipschitz character of $g$ and $F_1,\dots,F_h$, and the minimality of $(\Omega_1,\dots,\Omega_h)$, we have
\begin{equation*}
\begin{array}{ll}
m\left(|\Omega_1|-|\widetilde\Omega_1|\right)&\le g(F_1(\widetilde\Omega_1),F_2(\Omega_2),\dots,F_h(\Omega_h))-g(F_1(\Omega_1),F_2(\Omega_2),\dots,F_h(\Omega_h))\\
\\
&\le L\left(F_1(\widetilde\Omega_1)-F_1(\Omega_1)\right)\le CL\left(d_\g(\widetilde\Omega_1,\Omega_1)\right),
\end{array}
\end{equation*}
where $L$ is the Lipschitz constant of $g$ and $C$ the constant from Definition \ref{gammalip}. Repeating the argument for $\Omega_i$, we obtain that it is an energy subsolution with Lagrange multiplier $(CL)^{-1}m$. \end{proof}     
As a consequence, Theorem \ref{thss}, Proposition \ref{sep2qo} and Theorem \ref{lem130306} apply so we have all information about the perimeter of the cells and their interaction. In particular, there exists a family of open sets $\{D_1,\dots,D_h\}\subset\Dr$ such that
 $$\Omega_i\subset D_i,\ \forall i\in\{1,\dots,h\}\qquad\hbox{and}\qquad\cp(\Omega_i\cap D_j)=0,\ \forall i\neq j\in\{1,\dots,h\}.$$
 Moreover, $\Omega_i$ is a solution of the problem 
\begin{equation}\label{opft2e1}
\min\left\{F_i(\Omega):\ \Omega\subset D_i,\ \Omega\ \hbox{quasi-open},\ |\Omega|=|\Omega_i|\right\}.
\end{equation}

\begin{oss}
We note that Theorem \ref{opft2}  also holds in the case of \emph{subsolutions} of \eqref{opft1e1}. 
\end{oss}    
    
Here is a first example where Theorem \ref{opft2} applies.

\begin{cor}\label{lb2op}
Let $\Dr\subset\R^d$ be a bounded open set and $m>0$. Let $k_i\in \N$, $i=1,\dots,h$ and  $(\Om_1,\dots,\Om_h)$ be a solution of 
\begin{equation}\label{lb1partop}
\min\left\{\sum_{i=1}^h\lambda_{k_i}(\Omega_i)+m|\Omega_i|:\ \Omega_i\subset\Dr\ \hbox{quasi-open},\ \forall i;\ \Omega_i\cap\Omega_j=\emptyset,\forall i\neq j\right\}.
\end{equation}
Then, for every $i=1,\dots, h$ the quasi-open set $\Om_i$ is an energy subsolution. 
If, moreover, $k_i\in \{1,2\}$, then there exist open sets $\om_i\sq \Om_i$ such that $(\om_1,\dots,\om_h)$ is also a solution of \eqref{lb1partop}.
\end{cor}   
  \begin{proof}
 The fact that each $\Om_i$ is a subsolution relies on the $\g$-Lip property of the $k$-th eigenvalue. If  $k_i\in \{1,2\}$, we use the existence of an open set $D_i$ such that $\Om_i$ is solution of 
 \begin{equation}\label{opft2e1.1}
\min\left\{\lb_{k_i}(\Omega)+ m|\Om| :\ \Omega\subset D_i,\ \Omega\ \hbox{quasi-open}\right\}.
\end{equation}
If $k_i=1$, following \cite{brla}, the set $\Om_i$ is open. If $k_i=2$, we note that the functional $\lambda_2$ can be alternatively defined as
$$\lambda_2(\Omega)=\min\Big\{\max\left\{\lambda_1(\Omega_a),\lambda_1(\Omega_b)\right\}:\ \Omega_a,\Omega_b\subset\Omega\ \hbox{quasi-open},\ \Omega_a\cap\Omega_b=\emptyset\Big\}.$$
Thus, if $(\Omega_a,\Omega_b)\in[\mathcal{A}_{\cp}(D_i)]^2$ is a solution of \eqref{opft1e1} with $g(x_1,x_2)=\max\{x_1,x_2\}$ and $F_a=F_b=\lambda_1$, then the set $\Omega=\Omega_a\cup\Omega_b$ is a solution of \eqref{opft2e1.1}.  Now, the quasi-open sets $\Om_a$ and $\Om_b$ can be isolated by open sets $D_a$ and $D_b$. Thus, $\Om_a$ and $\Om_b$ minimize the first Dirichlet eigenvalue with a fixed measure constraint in $D_a$ and $D_b$, respectively. Relying again on the regularity result from \cite{brla}, we obtain that $\Omega_a$ and $\Omega_b$ are open sets.
 \end{proof}

In particular the following holds.
\begin{cor}\label{lb2op}
Let $\Dr\subset\R^d$ be a bounded open set and $m>0$. For every solution $\Omega$ of the problem 
\begin{equation}\label{lb2ope1}
\min\left\{\lambda_2(\Omega)+m|\Omega|:\ \Omega\ \hbox{quasi-open},\ \Omega\subset\Dr\right\},
\end{equation}
there exists an open set $\om\sq \Om$ which is also solution and has the same measure as $\Omega$. 
\end{cor}    

   We emphasize that not every solution of \eqref{lb2ope1} is an open set. In fact, if the set $D$ and the constant $m$ are suitably chosen, there is a family of  solutions obtained by erasing the nodal line   of the second eigenfunction associated on the "largest" optimal set. The eigenfunction itself does not change, while the shape does. 
   Moreover, the optimal set is equivalent to an open set in the sense of the Lebesgue measure.

   A somehow similar result for functionals involving higher eigenvalues holds for $m=0$ in dimension $2$. We note that the existence of an optimal open partition  was already proved in \cite{bobuou}.

    \begin{teo}
    Let $D\sq \R^2$ be a bounded, open and smooth set, let $m=0$ and $k_i\in \N$, $i=1,\dots,h$. Let  $(\Om_1,\dots,\Om_h)$ be a solution of 
\eqref{lb1partop}. There exists a solution $(\widetilde\Om_1, \dots, \widetilde\Om_h)$ consisting of open sets such that, for every $i=1,\dots,h$, 
$$\ \lb_{k_i}(\widetilde\Om_i)=\lb_{k_i}(\Om_i)= \lb_{k_i}(\Om_i\cap \widetilde\Om_i).$$
Moreover, every eigenfunction $u_{k_i} (\widetilde \Om_i)$ is H\"older continuous on $\overline D$.
    \end{teo}
    
    \begin{proof}
By \cite[Theorem 2.1]{bobuou}, we have that for each $\eps>0$ there are open sets $(A_1^\eps,\dots, A_h^\eps)$ such that $A_i^\eps\cap A_j^\eps=\emptyset$, for every $i\neq j\in\{1,\dots,h\}$ and $A_i^\eps$ $\gamma$-converges to $\Omega_i$, for every $i=1,\dots,h$. By choosing appropriate subsets of each $A_i^\eps$, we may suppose that the connected components of the open sets $A_1^\eps,\dots A_h^\eps$ are polygons. For each $A_i^\eps$ let $E_i^\eps\subset A_i^
\eps$ be a union of at most $k_i$ connected components of $A_i^\eps$ and such that $\lambda_{k_i}(E_i^\eps)=\lambda_{k_i}(A_i^\eps)$. By the compactness of the weak $\gamma$-convergence, we may suppose that $E_i^\eps$ weak $\gamma$-converges to some quasi-open set $\omega_i\subset\Omega_i$. Moreover, we have that $\omega_i\cap\widetilde\Omega_j=\emptyset$, for $i\neq j$, and 
$$\lambda_{k_i}(\omega_i)\le\liminf_{\eps\to0}\lambda_{k_i}(E_i^\eps)=\lim_{\eps\to0}\lambda_{k_i}(A_i^\eps)=\lambda_{k_i}(\Omega_i),$$
and, by the optimality of $\Omega_1,\dots,\Omega_h$, we have $\lambda_{k_i}(\omega_i)=\lambda_{k_i}(\Omega_i)$.  

We now enlarge each $E_i^\eps$ in order to obtain a partition which covers $\Dr$. We claim that for each $\eps$ there are disjoint open sets $F_1^\eps,\dots, F_h^\eps$ such that $E_i^\eps\subset F_i^\eps$, $F_i^\eps$ has at most $k_i$ connected components, $\Dr\cap\partial F_i^\eps$ is piecewise linear and $\overline \Dr=\bigcup_{i=1}^h\overline{F_i^\eps}$. One can obtain the family $F_1^\eps,\dots, F_h^\eps$ from $E_1^\eps,\dots, E_h^\eps$, considering all the connected components of $\Dr\setminus\left(\bigcup_{i=1}^h\overline{E_i^\eps}\right)$ and adding them, one by one, to one of the sets $E_1^\eps,\dots, E_h^\eps$, with which they have common boundary.
   We note that for every $i=1,\dots,h$, the number of connected components of $\R^2\setminus F_i^\eps$ is bounded uniformly in $\eps$. Thus, by Sverak's Theorem (see, for example, \cite[Theorem 4.7.1]{bubu05}), there are disjoint open sets $\widetilde\Omega_1,\dots,\widetilde\Omega_h$ such that $F_i^\eps$ $\gamma$-converges to $\widetilde\Omega_i$. Moreover, we have  $\omega_i\subset\widetilde\Omega_i$ and since, 
   $$\lambda_{k_i}(\widetilde\Omega_i)\le\liminf_{\eps\to0}\lambda_{k_i}(F_i^\eps)\le\liminf_{\eps\to0}\lambda_{k_i}(E_i^\eps)=\lambda_{k_i}(\Omega_i),$$
by the optimality of $\Omega_1,\dots,\Omega_h$, we have that $\lambda_{k_i}(\widetilde\Omega_i)=\lambda_{k_i}(\Omega_i)=\lambda_{k_i}(\omega_i)$.

Each eigenfunction belongs to $C^{0,\alpha} (\overline D)$ as a consequence of the fact that the sets $\R^2 \setminus \widetilde \Om_i$ have a finite number of connected components, hence they satisfy a uniform capacity density condition (see for instance \cite[Theorem 4.6.7]{bubu05}).
     
    \end{proof}

\section{Appendix: Proof of the Monotonicity Lemma}    
    
    The proof of Lemma \ref{mth3} follows the main steps and arguments of Theorem \ref{mth2}, for which we refer the reader to \cite{cajeke}. Nevertheless, the proof of   Lemma \ref{mth3} is simplified by the use of  the conclusion of Theorem \ref{mth2}. For the convenience of the reader, we use similar notations as in  \cite{cajeke}.  We report here only the technical difficulties brought by the absence of continuity of the functions $u_i$ and presence of the third phase. 
    
We start with recalling some preliminary results from \cite{cajeke}. For $i=1,2,3$, we use the notations 
\begin{equation}
A_i(r)=\int_{B_r}\frac{|\nabla u_i|^2}{|x|^{d-2}}\,dx,\qquad b_i(r)=\frac{1}{r^4}\int_{B_r}\frac{|\nabla u_i|^2}{|x|^{d-2}}\,dx.
\end{equation}   
We note that $A_i$ is increasing in $r$ and that $b_i$ is invariant under the rescaling $\widetilde u(x)=\frac{1}{r^2}u(xr)$.

\begin{lemma}\label{rem15}
There is a dimensional constant $C_d$ such that for each non-negative function $u\in H^1(\R^d)$ such that $\Delta u\ge-1$, we have 
\begin{equation*}
\int_{B_1}\frac{|\nabla u|^2}{|x|^{d-2}}\,dx\le C_d\left(1+\int_{B_2}u^2\,dx\right).
\end{equation*}
\end{lemma}   
\begin{proof}
See \cite[Remark 1.5]{cajeke}.
\end{proof}
%
%
 
\begin{lemma}\label{lemma24}
There are dimensional constants $C>0$ and $\eps>0$ such that if $u_i$, $i=1,2,3$ are as in Theorem \ref{mth3} and $A_i(r)\ge C$, for every $i=1,2,3$ and $r\in[1/4,1]$, then
 for every $r\in[1/4,1]$ we have
$$\frac{d}{dr}\left[\frac{A_1(r)A_2(r)A_3(r)}{r^{6+3\eps}}\right]\ge -C\left(\frac{1}{\sqrt{A_1(r)}}+\frac{1}{\sqrt{A_2(r)}}+\frac{1}{\sqrt{A_3(r)}}\right)\frac{A_1(r)A_2(r)A_3(r)}{r^{6+3\eps}}.$$
\end{lemma}     
\begin{proof}
A similar two-phase result is proved in \cite{cajeke} in the  framework of continuous functions.  We set, for $i=1,2,3$ and $r>0$,
$$B_i(r)=\int_{\partial B_r}|\nabla u_i|^2\,d\HH^{d-1}.$$
Then, computing the derivative as in \cite[Lemma 2.4]{cajeke}, it is sufficient to prove for almost every $r$
\begin{equation}\label{prf130306_3}
-\frac{6+3\eps}{r}+\frac{B_1(r)}{A_1(r)}+\frac{B_2(r)}{A_2(r)}+\frac{B_3(r)}{A_3(r)}\ge -C\left(\frac{1}{\sqrt{A_1(r)}}+\frac{1}{\sqrt{A_2(r)}}+\frac{1}{\sqrt{A_3(r)}}\right).
\end{equation}

We shall prove this inequality only for $r$ such that $B_i(r) <+\infty$ (which means that $u_i\lfloor_{\partial B_r}$ belongs to $H^1(\partial B_r)$). By a rescaling argument we can assume that  $r=1$.

   We first note that since $\Delta (u_i(x)+|x|^2/2d)\ge 0$, we have
$$\max_{x\in B_{1/2}}u_i(x)+|x|^2/2d\le C_d+C_d\int_{\partial B_1} u_i\,d\HH^{d-1},$$
and as a consequence
\begin{equation}\label{subs130305}
\int_{B_1}u_i|x|^{2-d}\le C_d+C_d\int_{\partial B_1} u_i\,d\HH^{d-1}\le C_d+C_d\left(\int_{\partial B_1} u_i^2\,d\HH^{d-1}\right)^{1/2}.
\end{equation}
 
We     now prove that if $C_d$ is large enough, then each set $\{u_i>0\}$ intersects the sphere $\partial B_1$, i.e. $\cp(\{u_i>0\}\cap\partial B_1)>0$. Indeed, suppose by absurd that $\cp(\{u_i>0\}\cap\partial B_1)=0$ and let $\widetilde u_i:=1_{B_1}u_i\in H^1_0(B_1)$. We have that $\Delta \widetilde u_i+1\ge 0$ and $\widetilde u_i\in H^1_0(B_{1+\eps})$ for every $\eps>0$. Up to an approximation in $H^1_0(B_{1+\eps})$, we can suppose that $\widetilde u_i\in C^\infty_c(B_{1+\eps})$. Using the fact that 
\begin{equation}\label{eq130306}
\Delta (\widetilde u_i^2)=2|\nabla\widetilde u_i|^2+2\widetilde u_i\Delta \widetilde u_i\ge 2|\nabla\widetilde u_i|^2-2\widetilde u_i,
\end{equation}
we obtain
\begin{equation*}
\begin{array}{ll}
2\int_{B_{1+\eps}}|\nabla \widetilde u_i|^2|x|^{2-d}\,dx&\le \int_{B_{1+\eps}}\left(2\widetilde u_i+\Delta (\widetilde u_i^2)\right)|x|^{d-2}\,dx\\
\\
&\le \int_{B_{1+\eps}}2\widetilde u_i|x|^{2-d}+\widetilde u_i^2\Delta (|x|^{d-2})\,dx\\
\\
&\le \int_{B_{1+\eps}}2\widetilde u_i|x|^{2-d}\,dx.
\end{array}
\end{equation*} 
Letting $\eps\to0$ and using \eqref{subs130305}, we obtain
\begin{equation}\label{subs130305b}
A_i(1)=\int_{B_1}|\nabla u_i|^2|x|^{2-d}\,dx\le C_d,
\end{equation}
which contradicts the hypothesis of the Lemma, for $C$ large enough.

   Since $\cp(\{u_i>0\}\cap\partial B_1)>0$, we have that $\cp_{d-1}(\{u_i>0\}\cap\partial B_1)>0$, for each $i=1,2,3$, where for any set $U\subset\partial B_1$ we denote with $\cp_{d-1}(U)$ the $(d-1)$-dimensional capacity on the sphere $\partial B_1$. 
   
    We next note that $\cp_{d-1}(\{u_i>0\}\cap\{u_j>0\};\partial B_1)=0$, for every $i\neq j\in\{1,2,3\}$. Indeed, since $|\{u_i>0\}\cap\{u_j>0\}|=0$ and $\{u_i>0\}\cap\{u_j>0\}$ is a quasi-open set in $\R^d$, we have $\cp(\{u_i>0\}\cap\{u_j>0\})=0$ and so $\HH^{d-1}(\partial B_1\cap\{u_i>0\}\cap\{u_j>0\})=0$. On the other hand, the restrictions of $u_i$ and $u_j$ on $\partial B_1$ are Sobolev functions. Thus the set $\partial B_1\cap\{u_i>0\}\cap\{u_j>0\}$, being quasi-open in $\partial B_1$ and of zero measure, is such that $\cp_{d-1}(\{u_i>0\}\cap\{u_j>0\};\partial B_1)=0$.

   Thus, for any $i\neq j\in\{1,2,3\}$, we have that $u_i$ is zero $\cp_{d-1}$-quasi-everywhere on the set $\{u_j>0\}$, which has a positive capacity on $\partial B_1$. Consequently, there is a constant $\lambda_i>0$ such that, for every $u\in H^1_0(\{u_i>0\}\cap\partial B_1)$, we have
$$\lambda_i\int_{\partial B_1}u^2\,d\HH^{d-1}\le \int_{\partial B_1}|\nabla_\tau u|^2\,d\HH^{d-1},$$   
where $\nabla_\tau$ is the tangential gradient on $\partial B_1$. In particular, we have    
$$\lambda_i\int_{\partial B_1}u_i^2\,d\HH^{d-1}\le \int_{\partial B_1}|\nabla_\tau u_i|^2\,d\HH^{d-1}\le \int_{\partial B_1}|\nabla u_i|^2\,d\HH^{d-1}=B_i(1).$$

Reasoning again as in \cite[Lemma 2.4]{cajeke}
we have that
\begin{equation}
2A_i(1)=2\int_{B_{}}|\nabla u_i|^2|x|^{2-d}\,dx\le C_d+C_d\sqrt{B_i(1)/\lambda_i}+ B_i(1)/\alpha_i,
\end{equation} 
where $\alpha_i>0$ satisfies
$$\alpha_i(\alpha_i+d-2)=\lambda_i.$$ 
Suppose first that there is some $i=1,2,3$, say $i=1$, such that $(6+3\eps)A_1(1)\le B_1(1)$. Then we have 
\begin{equation*}
-(6+3\eps)+\frac{B_1(1)}{A_1(1)}+\frac{B_2(1)}{A_2(1)}+\frac{B_3(1)}{A_3(1)}\ge -(6+3\eps)+\frac{B_1(1)}{A_1(1)}\ge 0,
\end{equation*}
 which proves \eqref{prf130306_3} and thus $(B)$.
 
 Assume now that for each $i=1,2,3$, we have $(6+3\eps)A_i(1)\ge B_i(1)$. Since, for each $i=1,2,3$ $A_i(1)\ge C_d$, we have 
\begin{equation*}
2A_i(1)\le C_d\sqrt{B_i(1)/\lambda_i}+ B_i(1)/\alpha_i\le C_d\sqrt{A_i(1)/\lambda_i}+ B_i(1)/\alpha_i.
\end{equation*} 
Moreover, $\alpha_i^2\le \lambda_i$, implies 
\begin{equation}
2\alpha_i A_i(1)\le  C_d\sqrt{A_i(1)}+ B_i(1).
\end{equation} 
Dividing both sides by $A_i(1)$ and summing for $i=1,2,3$, we obtain
\begin{equation*}
2(\alpha_1+\alpha_2+\alpha_3)\le C_d\sum_{i=1}^3\frac{1}{\sqrt{A_i(1)}}+\sum_{i=1}^3\frac{B_i(1)}{A_i(1)},
\end{equation*} 
and so, in order to prove \eqref{prf130306_3} and $(B)$, it is sufficient to prove that $\alpha_1+\alpha_2+\alpha_3\ge (6+3\eps)/2$. Let $\Omega^*_1,\Omega^*_2,\Omega^*_3\subset\partial B_1$ be the optimal quasi-open partition of the sphere $\partial B_1$ which is solving
\begin{equation}\label{minpb130306}
\min\left\{\alpha(\Omega_1)+\alpha(\Omega_2)+\alpha(\Omega_3):\ \Omega_i\subset\partial B_1\ \hbox{quasi-open},\forall i;\ \cp(\Omega_i\cap\Omega_j)=0, \forall i\neq j\right\},
\end{equation} 
where $\alpha(\Omega)$ is the unique positive real number such that $\lambda_1(\Omega)=\alpha(\Omega)(\alpha(\Omega)+d-2)$. We note that, as in \cite{cajeke}, $\alpha(\Omega^*_i)+\alpha(\Omega^*_j)\ge 2$, for $i\neq j$ and so summing on $i$ and $j$, we have
$$3\le \alpha(\Omega^*_1)+\alpha(\Omega^*_2)+\alpha(\Omega^*_3)\le \alpha_1+\alpha_2+\alpha_3.$$
Moreover, the first inequality is strict. Indeed, if this is not the case, we have that $\alpha(\Omega^*_1)+\alpha(\Omega^*_2)=2$ and so $|\partial B_1\sm (\Omega^*_1 \cup \Omega^*_2)|=0$, which would give that $\alpha (\Om^*_3)=+\infty$.
Choosing $\eps$ to be such that $6+3\eps$ is smaller than the minimum in \eqref{minpb130306}, the proof is concluded.  
\end{proof}

In what follows we will adopt the notation
$$A_i^k:=A_i(4^{-k}),\qquad b_i^k:=4^{4k}A_i(4^{-k}).$$
We prove a three-phase version of   \cite[Lemma 2.8]{cajeke}.

\begin{lemma}\label{lemma28}
There are dimensional constants $C_d>0$ and $\eps>0$ such that if the functions $u_i$, for $i=1,2,3$, are as in Theorem \ref{mth3} and $b_i^k\ge C_d$, for every $i=1,2,3$, then
$$4^{6+3\eps}A_1^{k+1}A_2^{k+1}A_3^{k+1}\le (1+\delta_k)A_1^kA_2^kA_3^k,$$
where 
$$\delta_k=C_d\left(\frac{1}{\sqrt{b_1^k}}+\frac{1}{\sqrt{b_2^k}}+\frac{1}{\sqrt{b_3^k}}\right).$$
\end{lemma}     
\begin{proof}
By rescaling, it is sufficient to consider the case $k=0$, in which we also have $b_i^0=A_i^0$, for $i=1,2,3$. 

   We consider two cases:
\begin{itemize}
\item Suppose that for some $i=1,2,3$, say $i=1$, we have $4^{2+3\eps}b_i^1\le b_i^0$. Then we have 
\begin{equation*}
\begin{array}{ll}
4^{6+3\eps}A_1^1A_2^1A_3^1=4^{2+3\eps}b_1^1A_2^1A_3^1\le b_1^0A_2^1A_3^1= A_1^0A_2^1A_3^1\le A_1^0A_2^0A_3^0.
\end{array}
\end{equation*} 

\item Suppose that for every $i=1,2,3$, we have $4^{2+3\eps}b_i^1\ge b_i^0$. Then $b_i^1\ge C_d$ for some $C_d$ large enough and so, we can apply Lemma \ref{lemma24}, obtaining that
\begin{equation*}
\begin{array}{ll}
\Phi'(r)&\ge -C\left(\frac{1}{\sqrt{A_1(r)}}+\frac{1}{\sqrt{A_2(r)}}+\frac{1}{\sqrt{A_3(r)}}\right)\Phi(r)\\
\\
&\ge -C4^{(2+3\eps)/2}\left(\frac{1}{\sqrt{b_1^0}}+\frac{1}{\sqrt{b_2^0}}+\frac{1}{\sqrt{b_3^0}}\right)\Phi(r),
\end{array}
\end{equation*} 

where $\Phi(r)=r^{-(6+3\eps)}A_1(r)A_2(r)A_3(r)$. Integrating for $r\in[1/4,1]$, we have 
$$\Phi(1/4)\ge\Phi(1)\exp\left(\frac{C}{\sqrt{b_1^0}}+\frac{C}{\sqrt{b_2^0}}+\frac{C}{\sqrt{b_3^0}}\right)\ge (1+\delta_0)\Phi(1).$$
\end{itemize}
\end{proof}

\begin{proof} (of Lemma \ref{mth3}) Let $M>0$ and let
$$S=\left\{k\in\N:\ 4^{(6+3\eps)k}A_1^kA_2^kA_3^k\le M\left(1+A_1^0+A_2^0+A_3^0\right)^2\right\}.$$
We will prove that if $\eps>0$ is small enough, then there is $M$ large enough such that for every $k\notin S$, we have 
$$4^{(6+3\eps)k}A_1^kA_2^kA_3^k\le CM\left(1+A_1^0+A_2^0+A_3^0\right)^2,$$
where $C$ is a constant depending on $d$ and $\eps$. 

We first note that if $k\notin S$, then we have
\begin{equation*}
\begin{array}{ll}
M(1+A_1^0+A_2^0+A_3^0)^2&\le 4^{(6+3\eps)k}A_1^{k}A_2^{k}A_3^{k}\\
\\
&\le 4^{-(2-3\eps)k}b_1^k 4^{4k}A_2^{k}A_3^{k}\\
\\
&\le 4^{-(2-3\eps)k}b_1^k C_d(1+A_1^0+A_2^0+A_3^0)^2,
\end{array}
\end{equation*}
and so $b_1^k\ge C_d^{-1}M4^{(2-3\eps)k}$, where $C_d$ is the constant from Theorem \ref{mth2}. Thus, choosing $\eps<2/3$ and $M>0$ large enough, we can suppose that, for every $i=1,2,3$, $b_i^k>C_d$, where $C_d$ is the constant from Lemma \ref{lemma28}. 

Suppose now that $L\in\N$ is such that $L\notin S$ and let 
$$l=\max\{k\in\N:\ k\in S\cap[0,L]\}<L,$$
where we note that the set $S\cap[0,L]$ is non-empty for large $M$, since for $k=0,1$, we can apply Theorem \ref{mth2}.
Applying Lemma \ref{lemma28}, for $k=l+1,\dots,L-1$ we obtain

\begin{equation}
\begin{array}{ll}
4^{(6+3\eps)L}A_1^{L}A_2^{L}A_3^{L}&\le\left(\prod_{k=l+1}^{L-1} (1+\delta_{k})\right)4^{(6+3\eps)(l+1)}A_1^{l+1}A_2^{l+1}A_3^{l+1}\\
\\
&\le\left(\prod_{k=l+1}^{L-1} (1+\delta_{k})\right)4^{(6+3\eps)(l+1)}A_1^{l}A_2^{l}A_3^{l}\\
\\
&\le\left(\prod_{k=l+1}^{L-1} (1+\delta_{k})\right)4^{6+3\eps}M\left(1+A_1^0+A_2^0+A_3^0\right)^2,
\end{array}
\end{equation}
where $\delta^k$ is the variable from Lemma \ref{lemma28}.

Now it is sufficient to notice that for $k=l+1,\dots, L-1$, the sequence $\delta_k$ is bounded by a geometric progression. Indeed, setting $\sigma=4^{-1+3\eps/2}<1$, we have that, for $k\notin S$, $\delta_k\le C\sigma^k$, which gives
\begin{equation}
\begin{array}{ll}
\prod_{k=l+1}^{L-1} (1+\delta_{k})&\le\prod_{k=l+1}^{L-1} (1+C\sigma^{k})\\
\\
&=\exp\left(\sum_{k=l+1}^{L-1}\log(1+C\sigma^k)\right)\\
\\
&\le \exp\left(C\sum_{k=l-1}^{L+1}\sigma^k)\right)\le \exp\left(\frac{C}{1-\sigma}\right),
\end{array}
\end{equation}
which concludes the proof.
\end{proof}


\bigskip
{\small\noindent
Dorin Bucur:
Laboratoire de Math\'ematiques (LAMA),
Universit\'e de Savoie\\
Campus Scientifique,
73376 Le-Bourget-Du-Lac - FRANCE\\
{\tt dorin.bucur@univ-savoie.fr}\\
{\tt http://www.lama.univ-savoie.fr/$\sim$bucur/}

\bigskip\noindent
Bozhidar Velichkov:
Scuola Normale Superiore di Pisa\\
Piazza dei Cavalieri 7, 56126 Pisa - ITALY\\
{\tt b.velichkov@sns.it}

\end{document}